\documentclass[11pt,reqno]{article}

\usepackage[usenames]{color}
\usepackage{graphicx}
\usepackage{amscd}
\usepackage[english]{babel}
\usepackage{color}
\usepackage{fullpage}
\usepackage{psfig}
\usepackage{graphics,amsmath,amssymb}
\usepackage{amsthm}
\usepackage{amsfonts}
\usepackage{latexsym}
\usepackage{epsf}
\usepackage{float}

\usepackage[colorlinks=true,
linkcolor=webgreen,
filecolor=webbrown,
citecolor=webgreen]{hyperref}

\definecolor{webgreen}{rgb}{0,.5,0}
\definecolor{webbrown}{rgb}{.6,0,0}

\newcommand{\seqnum}[1]{\href{http://oeis.org/#1}{\underline{#1}}}
\newcommand{\beql}[1]{\begin{equation}\label{#1}}
\newcommand{\eeq}{\end{equation}}
\newcommand{\eqn}[1]{(\ref{#1})}

\newcommand{\sA}{{\cal{A}}}
\newcommand{\sB}{{\cal{B}}}
\newcommand{\sC}{{\cal{C}}}

\newcommand{\sE}{{\cal{E}}}

\newcommand{\sP}{{\cal{P}}}
\newcommand{\sQ}{{\cal{Q}}}
\newcommand{\sY}{{\cal{Y}}}
\newcommand{\Snj}{\scriptstyle}
\newcommand{\TAIL}{\Omega}
\newcommand{\PROB}{\theta}
\newcommand{\EMPTY}{\epsilon}
\DeclareMathOperator{\CN}{cn}

\newtheorem{thm}{Theorem}{\bfseries}{\itshape}
\newtheorem{cor}[thm]{Corollary}{\bfseries}{\itshape}
\newtheorem{lem}[thm]{Lemma}{\bfseries}{\itshape}
{\bfseries}{\itshape}
\newtheorem{conj}[thm]{Conjecture}{\bfseries}{\itshape}

\newcommand{\LIMA}{{48}}
\newcommand{\LIMB}{{49}} 
\newcommand{\LIMC}{{80}}

\begin{document}
\theoremstyle{plain}

\begin{center}
{\large\bf On Curling Numbers of Integer Sequences} \\
\vspace*{+.2in}

Benjamin Chaffin, \\
Intel Processor Architecture, \\
2111 NE 25th Avenue, Hillsboro, OR 97124, USA, \\
Email: \href{mailto:chaffin@gmail.com}{\tt chaffin@gmail.com}, \\

\vspace*{+.1in}

John P. Linderman, \\
1028 Prospect Street, Westfield, NJ 07090, USA, \\
Email: \href{mailto:jpl.jpl@gmail.com}{\tt jpl.jpl@gmail.com}, \\

\vspace*{+.1in}

N. J. A. Sloane,\footnote{To whom correspondence should be addressed.} \\
The OEIS Foundation Inc., \\
11 South Adelaide Avenue, Highland Park, NJ 08904, USA, \\
Email: \href{mailto:njasloane@gmail.com}{\tt njasloane@gmail.com}, \\

and \\

Allan R. Wilks, \\
425 Ridgeview Avenue, Scotch Plains, NJ 07076, USA, \\
Email: \href{mailto:allan.wilks@gmail.com}{\tt allan.wilks@gmail.com}. \\

\vspace*{+.1in}


{\bf Abstract}\footnote{A preliminary report on the work in Section \ref{Sec2} was given in \cite{BS2009}}. 
\end{center}

Given a finite nonempty sequence $S$ of integers,
write it as $XY^k$, where $Y^k$ is a power of greatest exponent 
that is a suffix of $S$:
this $k$ is the {\em curling number} of $S$.
The {\em curling number conjecture} is that if
one starts with {\em any} initial sequence $S$,
and extends it by repeatedly appending the curling number of 
the current sequence, the sequence will eventually reach $1$.
The conjecture remains open. 
In this paper we discuss the special case when
$S$ consists just of $2$'s and $3$'s.
Even this case remains open, but we
determine how far a sequence consisting of $n$ $2$'s and $3$'s
can extend before reaching a $1$, conjecturally for $n \le 80$.
We investigate several related combinatorial problems,
such as finding $c(n,k)$, the number of binary sequences of length $n$ and curling 
number $k$, and $t(n,i)$, the number
of sequences of length $n$ which extend for $i$ steps before reaching a $1$.
A number of interesting combinatorial problems remain unsolved.

\section{The curling number conjecture}\label{Sec1}
Given a finite nonempty sequence $S$ of integers,
write it as $S = XY^k$, where $X$ and $Y$ are sequences of integers 
and $Y^k$ is a power of greatest exponent that is a suffix of $S$:
this $k$ is the {\em curling number} of $S$,
denoted by $\CN(S)$.
$X$ may be the empty sequence $\EMPTY $\,; there
may be several choices for $Y$, although the shortest such $Y$
which achieves $k$
(which as we shall see in \S\ref{SecCa} is primitive) is unique.

For example, if $S = 0\,1\,2\,2\,1\,2\,2\,1\,2\,2$,
we could write it as $XY^2$, where $X = 0\,1\,2\,2\,1\,2\,2\,1$
and $Y = 2$, or as $XY^3$, where $X = 0$
and $Y = 1\,2\,2$. The latter representation is to be
preferred, since it has $k=3$, and as $k=4$
is impossible, the curling number of this $S$ is $3$.

The following conjecture was stated
by van de Bult et al. \cite{GIJ}:

\begin{conj}\label{CNCconj}
The curling number conjecture.
If one starts with any initial sequence of integers $S$,
and extends it by repeatedly appending the curling number of 
the current sequence, the sequence will eventually reach $1$.
\end{conj}

In other words,
if $S_0 = S$ is any finite nonempty sequence of integers,
and we define  $S_{m+1}$ to be the concatenation
\beql{Eq1}
S_{m+1} ~:=~ S_m \, \CN(S_m) \mbox{~~for~} m \ge 0 \,,
\eeq
then the conjecture is that for some $t \ge 0$ we will have $\CN(S_t) = 1$. 
The smallest such $t$ is the {\em tail length} of $S_0$, denoted by 
$\tau (S_0)$ (and we set $\tau (S_0) = \infty$ if the conjecture is false).

For example, suppose we start with $S_0 = 2\,3\,2\,3$.
By taking $X = \EMPTY $, $Y = 2\,3$, we have $S_0 = Y^2$,
so $\CN(S_0) = 2$, and we get $S_1 = 2\,3\,2\,3\,2$.
By taking $X = 2$, $Y = 3\,2$ we 
get $\CN(S_1) = 2$, $S_2 = 2\,3\,2\,3\,2\,2$.
By taking $X = 2\,3\,2\,3$, $Y = 2$ we 
get $\CN(S_2) = 2$, $S_3 = 2\,3\,2\,3\,2\,2\,2$.
Again taking $X = 2\,3\,2\,3$, $Y = 2$ we 
get $\CN(S_3) = 3$, $S_4 = 2\,3\,2\,3\,2\,2\,2\,3$.
Now, unfortunately, it is impossible to write 
$S_4 = XY^k$ with $k>1$, so $\CN(S_4) = 1$,
$S_5 = 2\,3\,2\,3\,2\,2\,2\,3\,1$, and we have reached a $1$,
as predicted by the conjecture.
For this example, $\tau(S_0)=4$.
(If we continue the sequence from this point, it 
joins Gijswijt's sequence, discussed in \S\ref{Sec3}.)

Some of the proofs in van de Bult et al. \cite{GIJ} could be shortened and
the results strengthened if the conjecture were known to be true.
All the available evidence suggests that the conjecture {\em is} true,
but it has so far resisted all attempts to prove it.

In this paper we report on some extensive investigations into the case 
when the starting sequence consists of 2's and 3's (although
even in this special case the conjecture remains open).

In Section \ref{Sec2} we study how far a starting sequence consisting of $n$ 2's and 3's can extend
before reaching a 1. Call the maximum such length $\TAIL(n)$.
That is, $\TAIL(n)$ is the maximal value of the tail length $\tau (S_0)$ taken over
all sequences $S_0$ of 2's and 3's of length $n$.
We determine $\TAIL(n)$  for all $n \le \LIMA$, and conjecturally for
all $n \le \LIMC$ (Table \ref{Tabmu} and Figure \ref{Fig1}).
The data suggests some properties that should be possessed by especially
good starting sequences (Properties P2, P3, P4 in \S\ref{Sec2P}).
Although we have not found any algebraic construction
for good starting sequences, Section \ref{Sec2C}
describes a method which sometimes succeeds in building starting sequences
of greater length.
The algorithm which allowed us to extend the search to length $80$ is
discussed in \S\ref{Sec2D}.
We would not be surprised if the conjecture in this special case 
turns out to be a consequence of known results
on the unavoidability of patterns in long binary sequences---we discuss
this briefly in \S\ref{Sec2S}.

Section \ref{Sec3C} is devoted to the combinatorial question: what is 
the number $c(n,k)$  of binary sequences of length $n$ and
curling number $k$? This seems to be a surprisingly difficult problem,
and we have succeeded only in relating $c(n,k)$ to two
subsidiary quantities: $p(n,k)$,
the number of such sequences that are primitive, and $p'(n,k)$,
the number that are both primitive and robust (see  \S\ref{SecCa}).
The main results of this section are the formulas for $c(n,k)$ in Theorems \ref{ThCK} 
and \ref{ThSQRT}. With their help we are able to enumerate the curling
numbers of all binary sequences of length $n \le 104$.
The resulting table can be seen in entry \seqnum{A216955}\footnote{Throughout this article,
six-digit numbers prefixed by A refer to entries in the OEIS \cite{OEIS}.}
 in \cite{OEIS}. The number of binary sequences with curling number 1,
 $c(n,1)$ (\seqnum{A122536}),
 is especially interesting and is discussed in  \S\ref{SecCd}.
 Some further recurrences given there enable us to compute $c(n,1)$
 for $n \le 200$ (although we still do not know an explicit formula).
 We make frequent use of the classical Fine-Wilf theorem,
 and it and two other preliminary results are given in \S\ref{SecCb}.  
 The differences $d(n,k) ~:=~ 2\,c(n-1,k) ~-~ c(n,k)$ 
 show the structure of the $c(n,k)$ table more clearly than the numbers
 $c(n,k)$ themselves, and are the subject of \S\ref{SecCf}. 
 
 In Section \ref{Sec4}, we study the number $t(n,i)$ of sequences of length $n$ with 
 tail length $i$, where $0 \le i \le \TAIL(n)$. 
 By direct search we have determined $t(n,i)$
 for $n \le 48$ (\seqnum{A217209}), although without finding
 any recurrences (except for $t(n,0)$, which is the same as $c(n,1)$).
 The terms in each row of the $t(n,i)$ table occur in clumps, at least for $n \le 48$.
 In \S\ref{Sec3t} and \S\ref{Sec3a} we investigate some statistics 
 of the $t(n,i)$ table, although we are a long way from finding
 a model which explains the clumps.
  Sections \ref{Sec3r}, \ref{Sec3b}, \ref{Sec3p} discuss
 some combinatorial questions related to tail lengths.
 If the starting sequence $S_0$ is sufficiently long,
 it seems plausible that prefixing $S_0$ with a 2 or 3
 is unlikely to {\em decrease} the tail length.
 If one of these prefixes decreases the tail length,
 we call $S_0$ {\em rotten}, and if both prefixes 2 and 3 decrease
 the tail length we call it {\em doubly rotten}.
 Rotten sequences certainly exist, but up to length 34
 there are no doubly rotten sequences,
 and we conjecture than none exist of any length (see Conjecture \ref{ConjDR}).
 If this conjecture were true, it would explain a certain phenomenon
 that we observed in \S\ref{Sec2P},
 and it would also imply that $\TAIL(n+1) \ge \TAIL(n)$ for
all $n$, something that we do not know at present.

In Section \ref{Sec3} we briefly describe Gijswijt's sequence (\seqnum{A090822}),
which was the starting point for this investigation.
The last section summarizes the open problems
mentioned in the paper.
  
\paragraph{Notation}

Since the starting sequence $S$ can be any sequence of integers, it
seems appropriate in this paper to speak about ``sequences''
rather than ``words'' over some alphabet. However, we
will make use of certain terminology (such as ``prefix'', ``suffix'')
from formal language theory (cf. \cite{Loth}).

Sequences will be denoted by upper case Latin letters. 
$S^k$ means $SS \cdots S$, where $S$ is repeated $k$ times.
The length of $S$ is denoted by $|S|$.
$\EMPTY $ denotes the empty sequence.

Sets of sequences will be denoted by script letters (e.g., $\sC(n,k)$)
and their cardinalities by the corresponding lower case 
Latin letters (e.g., $c(n,k)$).
Greek letters and other lower case Latin letters will also denote numbers.
The symbol $\#$ denotes the cardinality of a set.

The curling number of $S$ is denoted by $\CN(S)$.
For a starting sequence 
$S_0 \, := \, s_1 \, s_2 \, \cdots \, s_n$
of length $n$, where the $s_i$ are arbitrary integers, we define
$S_{m+1}$ to be the concatenation $S_m \,\CN(S_m) = s_1 \, \cdots \ s_{n+m+1}$
for $m \ge 0$.
If $\CN(S_t)=1$ for some $t \ge 0$, then we call the smallest
such $t$ the {\em tail length} of $S_0$,
denoted by $\tau(S_0)$, and the corresponding sequence
$S^{(e)} \, := \, S_t = s_1 \, \cdots \ s_{n+t}$
is the {\em extension} of $S_0$.
If no such $t$ exists, then we set $\tau(S_0) = \infty$,
$S^{(e)} = S_{\infty}$ (and the curling number conjecture would be false).

\section{Sequences of $2$'s and $3$'s}\label{Sec2}

\begin{table}[!h]
$$
\begin{array}{|r|rrrrrrrrrrrr|}
\hline
     n   & 1 & 2 & 3 & 4 & 5 & 6 & 7 & 8 &   9 & 10 & 11 & 12 \\
\TAIL(n) & 0 & 2 & 2 & 4 & 4 & 8 & 8 & 58 & 59 & 60 & 112 & 112 \\
\hline
     n & 13    & 14 & 15 & 16 & 17 & 18 & 19 & 20 & 21 & 22 & 23 & 24 \\
\TAIL(n) & 112 & 118 & 118 & 118 & 118 & 118 & 119 & 119 & 119 & 120 & 120 & 120 \\
\hline
     n & 25 & 26 & 27 & 28 & 29 & 30 & 31 & 32 & 33 & 34 & 35 & 36 \\
\TAIL(n) & 120 & 120 & 120 & 120 & 120 & 120 & 120 & 120 & 120 & 120 & 120 & 120 \\
\hline
     n & 37 & 38 & 39 & 40 & 41 & 42 & 43 & 44 & 45 & 46 & 47 & 48 \\
\TAIL(n) & 120 & 120 & 120 & 120 & 120 & 120 & 120 & 120 & 120 & 120 & 120 & 131 \\
\hline
     n &  49 &  50 &  51 &  52 &  53 &  54 &  55 &  56 &  57 &  58 &  59 &  60 \\
\TAIL(n) & 131 & 131 & 131 & 131 & 131 & 131 & 131 & 131 & 131 & 131 & 131 & 131 \\
\hline
     n &  61 &  62 &  63 &  64 &  65 &  66 &  67 &  68 &  69 &  70 &  71 &  72 \\
\TAIL(n) & 131 & 131 & 131 & 131 & 131 & 131 & 131 & 132 & 132 & 132 & 132 & 132 \\
\hline
     n &  73 &  74 &  75 &  76 &  77 &  78 & 79  & 80 &  &  &  &  \\
\TAIL(n) & 132 & 132 & 132 & 133 & 173 & 173 & 173  & 173  &  &  &  & \\
\hline
\end{array}
$$
\caption{Lower bounds on $\TAIL(n)$, 
the maximal tail length that can be achieved before
a $1$ appears, for any starting sequence $S_0$ consisting of $n$ $2$'s and $3$'s.
Entries for $n \le \LIMA$ are known to be exact; the other
entries are conjectured to be exact.}
\label{Tabmu}
\end{table}

\subsection{ Maximal tail length $\TAIL(n)$ }\label{Sec2mu}
One way to approach the conjecture is to consider
the simplest nontrivial case, where the initial
sequence $S_0$ contains only $2$'s and $3$'s, and 
see how far such a sequence can extend using the
rule \eqn{Eq1} before reaching a $1$.
Perhaps if one were sufficiently clever,
one could invent a starting sequence that would
never reach 1, which would disprove the conjecture.
Of course it cannot reach a number greater than 3, either,
for the first time this happens the next term will be 1.
So the sequence must remain bounded between 2 and 3.
Unfortunately, even this apparently simple case 
has resisted our attempts to solve it.
At the end of this section (see \S\ref{Sec2S}) we will mention some slight 
evidence that suggests the conjecture is true. 
First we report on our numerical experiments.

Let $\TAIL(n)$ denote the maximal tail length that can be achieved before
a $1$ appears, for any starting sequence $S_0$ consisting of $n$ $2$'s and $3$'s.
If a 1 is never reached, we set $\TAIL(n) = \infty$.
The curling number conjecture would imply $\TAIL(n) < \infty$
for all $n$. 

By direct search, we have found $\TAIL(n)$ for all $n \le \LIMA$.
(The values for $n \le 30$ were given in \cite{GIJ}.)
The results are shown in Table \ref{Tabmu}
and Figure \ref{Fig1}, together with lower bounds
(which we conjecture are in fact equal to $\TAIL(n)$)
for $\LIMB \le n \le \LIMC$.
The values of $\TAIL(n)$ also form sequence \seqnum{A217208}
in \cite{OEIS}.

\begin{figure}[!h]
\centerline{\includegraphics[angle=270, width=5in]{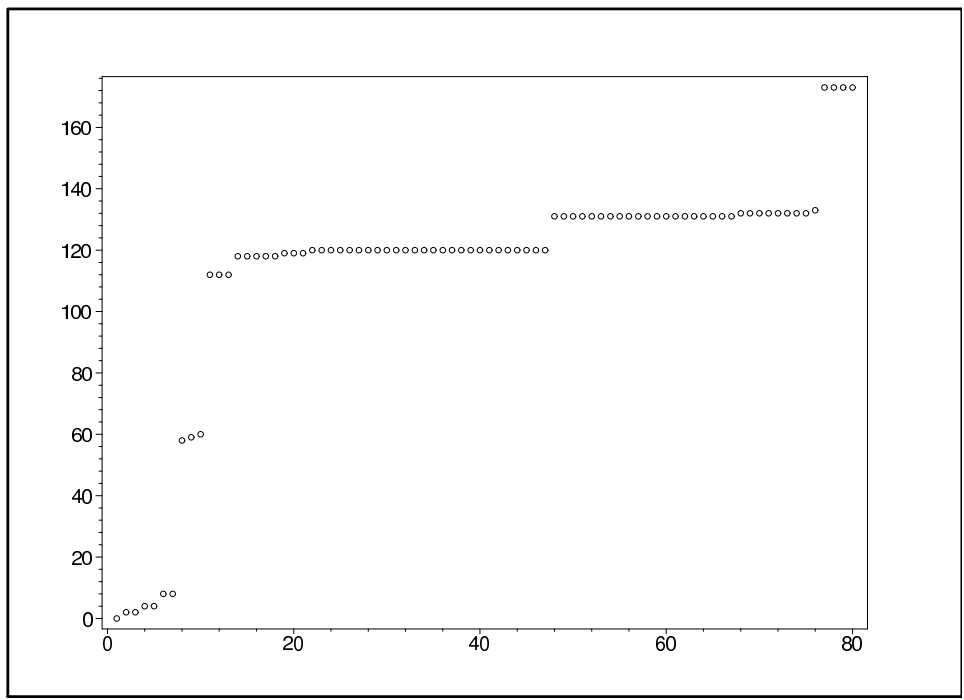}}
\caption{Scatter-plot of lower bounds on $\TAIL(n)$, 
the maximal tail length that can be achieved before
a $1$ appears, for any starting sequence $S_0$ consisting of $n$ $2$'s and $3$'s.
Entries for $n \le \LIMA$ are known to be exact; the other
entries are conjectured to be exact.}
\label{Fig1}
\end{figure}

In \cite{GIJ}, before we began computing $\TAIL(n)$,
we did not know how fast it would grow---would it
be a polynomial, exponential, or other function of $n$?
Even now we still do not know, since we have only limited
data. But up to $n = \LIMA$, and probably up to $n = \LIMC$,
$\TAIL(n)$ is a piecewise constant function of $n$.
There are occasional {\em jump points}, where
$\TAIL(n) > \TAIL(n-1)$, but in between jump points $\TAIL(n)$
does not change.
Of course this piecewise constant behavior is not
incompatible with polynomial or exponential growth,
if the jump points are close enough together, but 
up to $n = \LIMC$ this seems not to be the case. 
There are long stretches where $\TAIL(n)$ is flat.
A probabilistic argument will be given
in \S\ref{Sec4} which suggests (not very convincingly)
that, on the average, $\TAIL(n)$ may be roughly $c_1\,n$,
for a constant $c_1 \approx 1.34$.
Up to $n = \LIMB$, $\TAIL(n)$ never decreases,
although we cannot prove that this is always true (see \S\ref{Sec3r}).

The jump points are at 
$n = 1, 2, 4, 6, 8, 9, 10, 11, 14, 19, 22, 48$
and we believe the next three values are 68, 76 and 77
(\seqnum{A160766}).

\subsection{ Properties of good starting sequences }\label{Sec2P}
  From $n=2$ through $\LIMA$ (and probably through $n=\LIMC$)
the starting sequences $S_0$ which achieve $\TAIL(n)$
at the jump points are unique.
These especially good starting sequences are listed in 
Tables \ref{TabGood1} and \ref{TabGood2}.
For $2 \le n \le \LIMA$ (and probably
for $2 \le n \le \LIMC$)
these sequences $S_0$ also have the following properties:
\begin{list}{}{\setlength{\itemsep}{0.02in}}
\item
(P2) $S_0$ begins with 2.
\item
(P3) $S_0$ does not contain the subword 3\,3.
\item
(P4) $S_0$ contains no nonempty subword of the form $V^4$
(and in particular does not contain 2\,2\,2\,2).
\end{list}

\begin{table}[htb]
$$
\begin{array}{|l|l|}
\hline
n & \mbox{Starting sequence} \\
\hline 
1 & 2 \\
2 & 2\,2 \\
4 & 2\,3\,2\,3 \\
6 & 2\,2\,2\,3\,2\,2 \\
8 & 2\,3\,2\,2\,2\,3\,2\,3 \\
9 & 2\,2\,3\,2\,2\,2\,3\,2\,3 \\
10 & 2\,3\,2\,3\,2\,2\,2\,3\,2\,2 \\
11 & 2\,2\,3\,2\,3\,2\,2\,2\,3\,2\,2 \\
14 & 2\,2\,3\,2\,3\,2\,2\,2\,3\,2\,2\,3\,2\,3 \\
19 & 2\,2\,3\,2\,2\,3\,2\,3\,2\,2\,2\,3\,2\,2\,3\,2\,2\,3\,2 \\
22 & 2\,3\,2\,2\,3\,2\,2\,3\,2\,3\,2\,2\,2\,3\,2\,3\,2\,2\,3\,2\,2\,3 \\
48 & 2\,2\,3\,2\,2\,3\,2\,3\,2\,2\,2\,3\,2\,2\,2\,3\,2\,2\,3\,2\,2\,2\,3\,2\,2\,3\,2\,3\,2\,2\,2\,3\,2\,2\,2\,3\,2\,2\,3\,2\,2\,2\,3\,2\,2\,3\,2\,3 \\
\hline 
\end{array}
$$
\caption{Starting sequences consisting of $n$ 2's and 3's for which $\TAIL(n) > \TAIL(n-1)$,
complete for $1 \le n \le \LIMA$.}
\label{TabGood1}
\end{table}

\begin{table}[htb]
$$
\begin{array}{|l|l|}
\hline
n & \mbox{Starting sequence} \\
\hline
68 & 2\,2\,3\,2\,2\,3\,2\,2\,2\,3\,2\,2\,3\,2\,3\,2\,2\,2\,3\,2\,2\,2\,3\,2\,2\,3\,2\,2\,2\,3\,2\,2\,3\,2\,2\,2\,3\,2\,2\,3\,2\,3\,2\,2\,2\,3\,2\,2 \\
 & \,2\,3\,2\,2\,3\,2\,2\,2\,3\,2\,2\,3\,2\,2\,2\,3\,2\,2\,3\,2 \\
76 & 2\,3\,2\,2\,2\,3\,2\,2\,3\,2\,2\,2\,3\,2\,2\,3\,2\,3\,2\,2\,2\,3\,2\,2\,2\,3\,2\,2\,3\,2\,2\,3\,2\,2\,2\,3\,2\,2\,3\,2\,2\,2\,3\,2\,2\,3\,2\,3 \\
 & \,2\,2\,2\,3\,2\,2\,2\,3\,2\,2\,3\,2\,2\,3\,2\,2\,2\,3\,2\,2\,3\,2\,2\,2\,3\,2\,2\,3 \\
77 & 2\,2\,3\,2\,2\,2\,3\,2\,3\,2\,2\,2\,3\,2\,2\,2\,3\,2\,2\,3\,2\,2\,2\,3\,2\,2\,2\,3\,2\,3\,2\,2\,2\,3\,2\,2\,2\,3\,2\,2\,3\,2\,2\,2\,3\,2\,2\,2 \\
 & \,3\,2\,3\,2\,2\,2\,3\,2\,2\,2\,3\,2\,2\,3\,2\,3\,2\,2\,2\,3\,2\,2\,3\,2\,2\,2\,3\,2\,3 \\
\hline
\end{array}
$$
\caption{Conjectured to be the complete list of starting
sequences of $n$ 2's and 3's for which $\TAIL(n) > \TAIL(n-1)$,
for the range $\LIMB \le n \le \LIMC$.}
\label{TabGood2}
\end{table}

These are empirical observations. However, since they certainly hold
for the first $2^{\LIMB} -1$ choices for $S_0$,
we venture to make the following conjecture:
\begin{conj}\label{Conj33}
If a starting sequence $S_0$ of $2$'s and $3$'s of length $n \ge 2$ achieves
$\TAIL(n)$ with $\TAIL(n) > \TAIL(n-1)$, then $S_0$ is unique
and has properties P2, P3 and P4.
\end{conj}

We can at least prove one result about these especially
good starting sequences.
Let $S_0 \, = \, s_1 \, s_2 \, \cdots \, s_n$
be any sequence of integers with extension
$S^{(e)} = S_t = s_1 \, \cdots \ s_{n+t}$, where $\CN(S_t)=1$.
Call $S_0$ {\em weak} if each $S_r$ ($r=0, \ldots,t-1$)
can be written as $X  Y^{s_{n+r+1}}$ with $X \ne \EMPTY $.
In other words, $S_0$ is weak if the initial term $s_1$ is
not necessary for the computation
of the curling numbers $s_{n+1}, \ldots, s_{n+t}$.
This implies that $\tau(S_0) = \tau(s_2 \, \cdots \, s_n)$,
and establishes
\begin{lem}\label{LemWeak}
If a starting sequence $S_0$ of length $n \ge 2$ achieves
$\TAIL(n) > \TAIL(n-1)$, then $S_0$ is not weak.
\end{lem}

One further empirical observation is worth recording,
concerning the starting sequences between
jump points. Suppose $n_0$, $n_1$ are
consecutive jump points, so that
$$
\TAIL(n) ~=~ \TAIL(n-1) \quad \mbox{~for~} n_0 < n < n_1 \,,
$$
and $\TAIL(n) > \TAIL(n-1)$ at $n=n_0$ and $n_1$.
Then for $n \le \LIMA$ and conjecturally for $n \le \LIMC$,
if $n_0 < n < n_1$, one
can obtain a starting sequence that achieves $\TAIL(n)$
by taking the starting sequence of length $n_0$
and prefixing it by a ``neutral'' string of $n-n_0$
$2$'s and $3$'s
that do not get used in the computation of $\TAIL(n)$.
Although this is not surprising, we are unable to
prove that such neutral prefixes must always exist.
We return to this topic in \S\ref{Sec3r}.

The large gaps between the jump points at 22 and 48 and between 48 and 68 are
especially noteworthy.
In particular, we have
\beql{Eq3}
\TAIL(n) = 120 \mbox{~for~} 22 \le n \le 47 \,,
\eeq
and, conjecturally,
\beql{Eq49}
\TAIL(n) = 131 \mbox{~for~} 48 \le n \le 67 \,.
\eeq

The data shown in Tables \ref{Tabmu}, \ref{TabGood1}, \ref{TabGood2}
and Figure \ref{Fig1} for $n$ in the range $\LIMB$ to $\LIMC$
were obtained by computer search under
the assumption that the starting sequence has the properties
P3 and P4 mentioned above,
although without making any assumption about uniqueness. 
As it turned out, assuming P3 and P4, the best
starting sequences at the jump points are indeed unique
and start with 2.
Assuming P3 and P4 greatly reduces the number of starting sequences 
that must be considered.
For example, simply excluding sequences that contain
four consecutive $2$'s or four consecutive $3$'s
reduces the number of candidates of length $n$ from $2^n$ 
to a constant times $c_{2} ^n$, where $c_{2} = 1.839\cdots$
(cf. \seqnum{A135491}).
However, this by itself is not enough to enable us to reach $n=\LIMC$.
We discuss the algorithm that we used in more detail in \S\ref{Sec2D}.

We should emphasize that in the (we believe unlikely)
event that there are starting sequences of length $n$
with $\LIMB \le n \le \LIMC$ that achieve $\TAIL(n)$ but do not satisfy properties P3 and P4,
it is possible our conjecture that there
are jump points at lengths 68, 76, and 77 may be wrong,
and there may be better starting sequences
than those shown in Table \ref{TabGood2}.

\subsection{ A construction for larger $n$ }\label{Sec2C}
We have not succeeded in finding any algebraic constructions
for good starting sequences.  However,
one simple construction enables us to
obtain lower bounds on $\TAIL(n)$ for some larger values of $n$.
Let $S_0$ be a sequence of length $n$ that achieves $\TAIL(n)$,
and let $S^{(e)}$ be its extension of length $n+\TAIL(n)$.
Then in some cases the starting sequence $S^{(e)}  S_0$
will extend to $S^{(e)}  S^{(e)}  2$ and beyond before reaching a $1$.
For example, taking $S_0$ to be the length $48$ 
sequence in Table \ref{TabGood1},
the sequence $S^{(e)} S_0$ has length $179+48=227$ and extends to 
a total length of $596$ before reaching a $1$,
showing that $\TAIL(227) \ge 369$.

\subsection{ Computational details }\label{Sec2D}
Our results are complete for $n \le \LIMA$ and are probably complete through
$n = \LIMC$.
In order to extend the search this far,
the algorithms used were specifically tuned to the case of 
sequences of 2's and 3's. 
There is no easy way (as far as we know) to avoid the basic process of 
computing the extension of $S$ (compute $\CN(S)$, 
append it to $S$, and repeat until $\CN(S)=1$), and 
so the focus is on computing $\CN(S)$ quickly. 
In the following discussion we assume that $S$ has length at least 36.
The first step is brute force: look up the curling number $\CN(s_{n-35}\,\cdots\,s_n)$
in a table. Two bits are sufficient to record $\CN$,
since we only care about whether it is 1, 2, 3 or $\ge 4$; 
at two bits per entry, this table occupies 16 gigabytes. 
This provides a lower bound on $\CN(S)$, 
and also gives a lower bound for the length of the 
repeated substring $Y$ which maximizes $\CN(S)$. 
For example, if 
$\CN(s_{n-35}\,\cdots\,s_n)=1$,
then any $Y$ which gives $\CN(S)>1$ must be at least 19 digits long, 
or there would have been two copies within the last 36 digits of S.

There is also an upper bound on the length of $Y$. 
Since we are looking for that $Y$ which maximizes $\CN(S)$, 
we are only interested in $Y$'s which could be repeated more times 
than the current best known value of $\CN(S)$. 
For example, if we know $\CN(S) \ge 3$, 
then we only want a $Y$ which is repeated four times, 
and so we only need consider lengths up to the length of $S$ divided by 4.

We now consider the last $n$ digits of $S$ as a candidate for $Y$, 
for all values of $n$ between the lower and upper bounds.
The sequences are represented as 128-bit binary numbers, 
and so looking for repetitions of $Y$ can be done with bit manipulation. 
A few shifts and OR's generate 4 copies of $Y$, 
or as many as will fit in 128 bits. 
Then an XOR finds digits in which this differs from $S$, 
a bit scan locates the index of the first difference, 
and we can divide by the length of $Y$ to find how many times this 
$Y$ is repeated. (In fact, all divisions are done with precomputed tables.) 
If some $Y$ increases the best known value for $\CN(S)$, 
then the upper bound on the length of $Y$ can be revised downwards. 
If we reach 128 digits and are still going, 
we resort to a slow string-based routine. 
In practice this slow routine accounts for less than 
1\% of the program's execution time.

To compute the conjectured values up to length 80, we exclude (most) 
strings containing 3\,3 or a subword $V^4$. 
Obviously we cannot check all $2^{80}$ strings to see if they 
violate one of these conditions, 
so we need an efficient way to avoid considering them at all. 
To do this, we compute a $256 \times 256$ table which lists, 
for every string of length 8, all the length-8 strings which could legally follow it. 
We then construct $S$ recursively in 8-digit blocks, 
ensuring that the rules are not broken within any two consecutive blocks. 
This is not perfect (it will allow a $V^4$ to slip by if $V$ is 9 digits long, 
for example), 
but it efficiently eliminates the vast majority of undesirable cases.

\subsection{ Unavoidable regularities }\label{Sec2S}
One reason we think the curling number conjecture may be true,
at least in the special case of sequences of 2's and 3's,
is that there are several theorems in formal
language theory about the inevitability of regularities
in long binary strings. A classical example is Shirshov's theorem
\cite[Theorem 7.1.4]{Loth}, \cite[Theorem 2.4.3]{LuVa}.
Unfortunately that does not quite do what we need, but it does 
offer hope that a proof along these lines may exist.
Lyndon's theorem \cite[p.~67]{Loth} is another example.
Suppose we have a very long sequence of 2's and 3's generated
by \eqn{Eq1}, and consider its canonical decomposition
into Lyndon words. There are relatively few
Lyndon words that are possible (e.g., 2\,2\,2\,2 is forbidden),
but since this attack has not yet led to a contradiction
we shall say no more about it.

\section{ Number of binary sequences with given curling number }\label{Sec3C}

In this section we study the number $c(n,k)$ of binary sequences of length $n$
and curling number $k$. 
For consistency with the other sections, we 
continue to consider sequences of 2's and 3's,
although for this question any alphabet of size 2 (such as \{0,1\})
would do equally well.

\subsection{ Primitive and robust sequences }\label{SecCa}
A sequence $S$ is {\em imprimitive} (or {\em periodic}) if it is equal to $T^i$ for
some sequence $T$ and an integer $i \ge 2$. Otherwise, $S$ is {\em primitive}  \cite[p.~7]{Loth}.

\begin{lem}\label{LemPi}
Suppose $S$ has curling number $k$. Then $S$ can be written as
$X Y^k$, possibly in several ways. The shortest such $Y$ is primitive and unique,
and has curling number $<k$ if $k>1$, curling number $1$ if $k=1$.
\end{lem}
\begin{proof}
Consider all possible ways of writing $S=X Y^k$, and let $\sY$ denote
the set of such $Y$'s of minimal length.
Every $Y \in \sY$ is primitive, for if $Y=T^i$, $i \ge 2$, then
$S = X T^{ik}$, and $\CN(S) \ge ik > k$,
contradicting the definition of $\sY$.
To establish uniqueness, we observe that $S = X_1 {Y_1}^k = X_2 {Y_2}^k$
with $|Y_1| = |Y_2|$ implies $Y_1=Y_2$.
If $k>1$ and $Y \in \sY$ has curling number 
$c \ge k > 1$, say $Y = U V^c$, $|V| \ge 1$,
then $S = X(UV^c)^k = X' V^c$ with $c \ge k$, $|V| < |Y|$, a contradiction.
Finally, if $k=1$, certainly $Y$ cannot have
curling number greater than 1, or $S$ would too.
\end{proof}
We denote the length of this shortest $Y$ by $\pi$.
We let $\sC(n,k,\pi)$ (for $n \ge 1$, $1 \le k \le n$, $1 \le \pi \le n$) denote the set of all $S$ with the
given values of $n$, $k$, and $\pi$,
$c(n,k,\pi) := \#\sC(n,k,\pi)$,
$\sC (n,k) := \bigcup_{\pi=1}^{\lfloor n/k \rfloor} \sC(n,k,\pi)$,
and $c(n,k) := \#\sC(n,k) = \sum_{\pi=1}^{\lfloor n/k \rfloor} c(n,k,\pi)$.

If $S$ has curling number 1 then the shortest $Y$ for
which $S = XY$ is simply the last term of $S$, so $\pi=1$
and $S \in \sC(n,1,1)$. The sets $\sC(n,1,\pi)$ for $\pi > 1$ are empty.

We let $\sP(n,k)$ (for $1 \le k \le n$) denote the subset of primitive $S \in \sC(n,k)$,
and $p(n,k) := \#\sP(n,k)$.
Note that $\sC(n,1) = \sP(n,1)$, since curling number 1 implies primitive.

Also let $\sQ (n,k) := \bigcup_{i=1}^{k} \sP(n,i)$ (for $1 \le k \le n$) denote
the set of primitive sequences with curling
number at most $k$, and 
$q(n,k) := \#\sQ(n,k)  = \sum_{i=1}^{k} p(n,i)$.
We also set $q(n,0) := 0$ and $q(n,k) := q(n,n)$ for $k>n$.
By definition, $q(n,n)$ is the total
number of aperiodic binary sequences of length $n$,
and it is well known (\cite{GR61}; see also entry \seqnum{A217943} in \cite{OEIS}) 
 that
\beql{Eqalpha}
q(n,n) ~=~ \sum_{d \mid n} \mu \left(\frac{n}{d} \right) 2^d \,,
\eeq
where $\mu$ is the M\"obius function ($q(n,n)$ is sequence \seqnum{A027375}).

Call $S \in \sP (n,k)$ {\em robust} if no proper suffix
of $S^{k+1}$ has curling number $\ge k+1$.
Examples of non-robust sequences first appear
at length 5, where $S ~=~ 3\,2\,2\,3\,2
\in \sC(5,1)$ is not robust since
$$
S^2 ~=~ 3\,2\,2\,3\,2 ~ 3\,2\,2\,3\,2
$$
has the suffix $(2\,3\,2)^2$.
At length 8 there are examples with $k=2$, such as $S ~=~ 3\,2\,2\,3\,2\,2\,3\,2$,
for which $S^3$ has the suffix $(2\,3\,2)^3$.
Let $\sP'(n,k)$ denote the subset of robust $S \in \sP(n,k)$,
and let $p'(n,k) := \#\sP'(n,k)$.

Tables \ref{Tabcnk}, \ref{Tabpnk}, \ref{Tabqnk},
and \ref{Tabppnk} show the initial values
of $c(n,k)$, $p(n,k)$, $q(n,k)$, and $p'(n,k)$, respectively.
There are far fewer non-robust sequences than robust, and their
numbers are shown in Table \ref{Tabsnk}.

\begin{table}[thbp]
$$
\begin{array}{|r|rrrrrrrrrrrr|}
\hline
n\backslash k & 1 & 2 & 3 & 4 & 5 & 6 & 7 & 8 & 9 & 10 & 11 & 12 \\
\hline
1 & 2 &   &   &   &   &   &   &   &   &    &    &   \\
2 & 2 & 2 &   &   &   &   &   &   &   &    &    &   \\
3 & 4 & 2 & 2 &   &   &   &   &   &   &    &    &   \\
4 & 6 & 6 & 2 & 2 &   &   &   &   &   &    &    &   \\
5 & 12 & 12 & 4 & 2 & 2 &   &   &   &   &    &    &   \\
6 & 20 & 26 & 10 & 4 & 2 & 2 &   &   &   &   &    & \\
7 & 40 & 52 & 20 & 8 & 4 & 2 & 2 &   &   &   &   &   \\
8 & 74 & 110 & 38 & 18 & 8 & 4 & 2 & 2 & & & & \\
9 & 148 & 214 & 82 & 36 & 16 & 8 & 4 & 2 & 2 & & & \\
10 & 286 & 438 & 164 & 70 & 34 & 16 & 8 & 4 & 2 & 2 & &\\
11 & 572 & 876 & 328 & 140 & 68 & 32 & 16 & 8 & 4 & 2 & 2 &\\
12 & 1124 & 1762 & 660 & 286 & 134 & 66 & 32 & 16 & 8 & 4 & 2 & 2\\
\hline
\end{array}
$$
\caption{Table of $c(n,k)$, the number of binary sequences of length $n$
and curling number $k$, for $1 \le k \le n$ and $n \le 12$ 
(for an extended table see \seqnum{A216955}).}
\label{Tabcnk}
\end{table}

\begin{table}[thbp]
$$
\begin{array}{|r|rrrrrrrrrrrr|}
\hline
n\backslash k & 1 & 2 & 3 & 4 & 5 & 6 & 7 & 8 & 9 & 10 & 11 & 12 \\
\hline
1 & 2 &  &  &  &  &  &  &  &  &  &  &  \\
2 & 2 & 0  & &  &  &  &  &  &  &  &  &  \\
3 & 4 & 2 & 0  & &  &  &  &  &  &  &  &  \\
4 & 6 & 4 & 2 & 0  & &  &  &  &  &  &  &  \\
5 & 12 & 12 & 4 & 2 & 0  & &  &  &  &  &  &  \\
6 & 20 & 20 & 8 & 4 & 2 & 0  & &  &  &  &  &  \\
7 & 40 & 52 & 20 & 8 & 4 & 2 & 0  & &  &  &  &  \\
8 & 74 & 100 & 36 & 16 & 8 & 4 & 2 & 0  & &  &  &  \\
9 & 148 & 214 & 76 & 36 & 16 & 8 & 4 & 2 & 0  & &  &  \\
10 & 286 & 414 & 160 & 68 & 32 & 16 & 8 & 4 & 2 & 0  & &  \\
11 & 572 & 876 & 328 & 140 & 68 & 32 & 16 & 8 & 4 & 2 & 0  & \\
12 & 1124 & 1722 & 640 & 276 & 132 & 64 & 32 & 16 & 8 & 4 & 2 & 0 \\
\hline
\end{array}
$$
\caption{Table of $p(n,k)$, the number of primitive binary sequences of length $n$
and curling number $k$, for $1 \le k \le n$ and $n \le 12$ 
(for an extended table see \seqnum{A218869}).}
\label{Tabpnk}
\end{table}

\begin{table}[thbp]
$$
\begin{array}{|r|rrrrrrrrrrrr|}
\hline
n\backslash k & 1 & 2 & 3 & 4 & 5 & 6 & 7 & 8 & 9 & 10 & 11 & 12 \\
\hline
1 & 2 & & & & & & & & & & & \\
2 & 2 & 2 & & & & & & & & & & \\
3 & 4 & 6 & 6 & & & & & & & & & \\
4 & 6 & 10 & 12 & 12 & & & & & & & & \\
5 & 12 & 24 & 28 & 30 & 30 & & & & & & & \\
6 & 20 & 40 & 48 & 52 & 54 & 54 & & & & & & \\
7 & 40 & 92 & 112 & 120 & 124 & 126 & 126 & & & & & \\
8 & 74 & 174 & 210 & 226 & 234 & 238 & 240 & 240 & & & & \\
9 & 148 & 362 & 438 & 474 & 490 & 498 & 502 & 504 & 504 & & & \\
1 & 286 & 700 & 860 & 928 & 960 & 976 & 984 & 988 & 990 & 990 & & \\
11 & 572 & 1448 & 1776 & 1916 & 1984 & 2016 & 2032 & 2040 & 2044 & 2046 & 2046 & \\
12 & 1124 & 2846 & 3486 & 3762 & 3894 & 3958 & 3990 & 4006 & 4014 & 4018 & 4020 & 4020 \\
\hline
\end{array}
$$
\caption{Table of $q(n,k)$, the number of primitive binary sequences of length $n$
and curling number at most $k$, for $1 \le k \le n$ and $n \le 12$ 
(for an extended table see \seqnum{A218870}).}
\label{Tabqnk}
\end{table}

\begin{table}[thbp]
$$
\begin{array}{|r|rrrrrrrrrrrr|}
\hline
n\backslash k & 1 & 2 & 3 & 4 & 5 & 6 & 7 & 8 & 9 & 10 & 11 & 12 \\
\hline
1 & 2 & & & & & & & & & & & \\
2 & 2 & 0 & & & & & & & & & & \\
3 & 4 & 2 & 0  & & & & & & & & & \\
4 & 6 & 4 & 2 & 0  & & & & & & & & \\
5 & 10 & 12 & 4 & 2 & 0  & & & & & & & \\
6 & 20 & 20 & 8 & 4 & 2 & 0  & & & & & & \\
7 & 36 & 52 & 20 & 8 & 4 & 2 & 0  & & & & & \\
8 & 72 & 98 & 36 & 16 & 8 & 4 & 2 & 0  & & & & \\
9 & 142 & 214 & 76 & 36 & 16 & 8 & 4 & 2 & 0  & & & \\
10 & 280 & 414 & 160 & 68 & 32 & 16 & 8 & 4 & 2 & 0  & & \\
11 & 560 & 870 & 326 & 140 & 68 & 32 & 16 & 8 & 4 & 2 & 0  & \\
12 & 1114 & 1720 & 640 & 276 & 132 & 64 & 32 & 16 & 8 & 4 & 2 & 0 \\
\hline
\end{array}
$$
\caption{Table of $p'(n,k)$, the number of robust primitive binary sequences of length $n$
and curling number $k$, for $1 \le k \le n$ and $n \le 12$ 
(for an extended table see \seqnum{A218875}).}
\label{Tabppnk}
\end{table}

\begin{table}[thbp]
$$
\begin{array}{|r|rrrrrrrrrrrr|}
\hline
n\backslash k & 1 & 2 & 3 & 4 & 5 & 6 & 7 & 8 & 9 & 10 & 11 & 12 \\
\hline
1 & 0 & & & & & & & & & & & \\
2 & 0 & 0  & & & & & & & & & & \\
3 & 0 & 0 & 0  & & & & & & & & & \\
4 & 0 & 0 & 0 & 0  & & & & & & & & \\
5 & 2 & 0 & 0 & 0 & 0  & & & & & & & \\
6 & 0 & 0 & 0 & 0 & 0 & 0  & & & & & & \\
7 & 4 & 0 & 0 & 0 & 0 & 0 & 0  & & & & & \\
8 & 2 & 2 & 0 & 0 & 0 & 0 & 0 & 0  & & & & \\
9 & 6 & 0 & 0 & 0 & 0 & 0 & 0 & 0 & 0  & & & \\
10 & 6 & 0 & 0 & 0 & 0 & 0 & 0 & 0 & 0 & 0  & & \\
11 & 12 & 6 & 2 & 0 & 0 & 0 & 0 & 0 & 0 & 0 & 0  & \\
12 & 10 & 2 & 0 & 0 & 0 & 0 & 0 & 0 & 0 & 0 & 0 & 0 \\
\hline
\end{array}
$$
\caption{The numbers $p(n,k)-p'(n,k)$ of non-robust primitive 
binary sequences of length $n$
and curling number $k$, for $1 \le k \le n$ and $n \le 12$ 
(for an extended table see \seqnum{A218876}).}
\label{Tabsnk}
\end{table}

\subsection{ Three preliminary theorems }\label{SecCb}

The classical Fine-Wilf theorem 
(\cite{FW65}; \cite[p.~13]{AS2003}, \cite{LS67}, \cite[p.~10]{Loth}) 
turns out to be very useful for studying curling numbers.

\begin{thm}\label{ThFW}
$($Fine and Wilf$)$
If sequences $S = X^i$ and $T = Y^j$ have a common suffix $U$
of length 
\beql{EqFW}
|U| ~\ge~ |X| ~+~ |Y| ~-~ \gcd(|X|,|Y|)\,,
\eeq
then, for some sequence $Z$ and integers $g$, $h$, we have
$X=Z^g$, $Y=Z^h$, $|Z| = \gcd(|X|,|Y|)$.
\end{thm}
In most applications all we will need is $|U| \ge |X| + |Y| - 1$,
rather than \eqn{EqFW} itself.
 
 There is an equivalent definition of robustness that is easier to check.
 
 \begin{thm}\label{ThRobust}
 If $S \in \sP(n,k)$ is not robust, implying that $S^{k+1}$ has a proper suffix $T^{k+1}$
 for some $T$, then $T^{k+1}$ is in fact a proper suffix of $S^2$.
 \end{thm}
 \begin{proof}
 The assertion is trivially true if $k=1$, so we assume $k \ge 2$.
 The hypotheses imply $t := |T| < n$. 
 Now $S^{k+1}$ and $T^{k+1}$ have a common suffix of length $(k+1)t$.
 If it were the case that $(k+1)t \ge n+t-1$, by Theorem \ref{ThFW} we would
 have $S=Z^g$, $T=Z^h$, for some $Z, g, h$ with $g>h$, implying $g \ge 2$
 and so $S$ would be imprimitive, a contradiction.
 So $(k+1)t < n+t-1 <2n$, as required.
 \end{proof}
 
 It follows that $S \in \sP(n,k)$ is robust if and only if no proper 
 suffix of $S^2$ has curling number $k+1$. This greatly
 simplifies the computation of the numbers $p(n,k)$.
 
A trivial but useful observation is that
prefixing a sequence with a single number
cannot increase the curling number by more than 1: 
\begin{thm}\label{Lem78.4}
If $S \in \sC(n,k)$ then $2S$ $($and equally $3S)$
is in either $\sC(n+1,k)$ or $\sC(n+1,k+1)$.
\end{thm}
\begin{proof}
If, for example, $2S \in \sC(n+1,l)$ with $l \ge k+2$,
then $2S = U\,V^l$ for some $U, V, l$, and
$V^{l-1}$ (at least) is a suffix of $S$, contradicting
the fact that $S$ has curling number $k$.
\end{proof}

\subsection{ A recurrence for $c(n,k)$ }\label{SecCc}
The first main theorem of this section expresses the $n$-th row of the $c(n,k)$ table
in terms of the $(n-1)$st row and much earlier rows of the $p(n,k)$ and $p'(n,k)$ 
tables.

\begin{thm}\label{ThCK}
The numbers $c(n,k)$ have the following properties:
$c(n,k)=0$ for $n \le k-1$,
$c(n,k)=2$ for $n=k$ and $k+1$, and, for $n \ge k+2$, 
\begin{align}\label{EqCK1}
c(n,k) & ~=~ 2 \, c(n-1,k) \nonumber \\
 & ~ + ~ [k \mid n] ~ \left( p'\left(\frac{n}{k},k-1\right) ~+~ 
 q\left( \frac{n}{k},k-2 \right) 
  \right) \nonumber \\
 & ~ - ~ [k+1 \mid n] ~ \left( p'\left(\frac{n}{k+1},k\right) ~+~ 
 q\left( \frac{n}{k+1},k-1 \right) 
  \right)\,, \nonumber \\
\end{align}
where the Iverson bracket $[R]$ is $1$ if the relation $R$ is true, $0$ otherwise.

\end{thm}

\begin{proof}
We assume $k \ge 1$ and $n \ge k+2$.
Suppose $S \in \sC(n,k)$ and let $T$ denote $S$ with its left-most
term deleted. We consider the cases $\CN(T)=k$ and $\CN(T)<k$ separately.

In the first case, if $T$ is any sequence in $\sC(n-1,k)$,
and $S$ is $2T$ or $3T$, then, by Theorem \ref{Lem78.4},
$S$ is in either $\sC(n,k)$ or $\sC(n,k+1)$.
So we will obtain $2 c(n-1,k)$ sequences in $\sC(n,k)$, except that we
must exclude from the count those $T \in \sC(n-1,k)$
with the property that $2T$ or $3T = V^{k+1}$ for some
primitive $V$ of length $n/(k+1)$. This can only happen when $n$ is a multiple of $k+1$.
These $V$'s are primitive sequences of length $n/(k+1)$, with curling number
$l \le k$, and are such that no proper suffix of $V^{k+1}$ has curling number greater than $k$.
If $l=k$, the number of such $V$'s is (by definition) 
$p'(n/(k+1),k )$.
On the other hand, if $1 \le l \le k-1$, any 
$V \in \sP ( n/(k+1),l )$ has the property that no proper suffix of $V^{k+1}$
has curling number greater than $k$ (and the number of these is $p  (n/(k+1),l ))$.
This follows from the Fine-Wilf theorem (Theorem \ref{ThFW}).
For if $V^{k+1}$ has a proper suffix of the form $U^{k+1}$, then these
two sequences overlap in the last $(k+1)u$ terms, where $u = |U|$, and also $u<v$,
where $v = |V| = n/(k+1)$. Since $V$ has curling number $l<k$, the right-most $k$
copies of $U$ are not a suffix of $V$, and so $ku>v$.
This implies
\beql{EqFW2}
(k+1)u ~\ge~ v+u-1\,,
\eeq
and so by Theorem \ref{ThFW}, $V=Z^g$, $U=Z^h$, $h<g$, $g \ge 2$.
But $V^k = Z^{2g}$ is a suffix of $T$, so $\CN(T) \ge 2k > k$, a contradiction.
(Further applications of the Fine-Wilf theorem will follow
this same pattern, and we will not give as much detail.)

In the second case we must consider sequences $S = V^k$ where $\CN(T) < k$.
Now $n$ must be a multiple of $k$, and $V \in \sP(n/k,l)$ for $1 \le l \le k-1$
is such that no proper suffix of $V^k$ has curling number $k$.
If $l=k-1$, the number of such $V$'s is (by definition) $p'(n/k,k-1)$.
On the other hand, if $1 \le l \le k-2$, the condition that
no proper suffix of $V^k$ has curling number $k$ follows from the Fine-Wilf theorem
by an argument similar to that given above (except that $k+1$
is replaced by $k$), and the number is $p(n/k,l)$.
This completes the proof of the theorem.
\end{proof}

\subsection{ Sequences with curling number 1}\label{SecCd}
For the purpose of investigating the curling number conjecture,
we are particularly interested in the first three
columns of the $c(n,k)$ table, since they determine the probabilities
that a random sequence of 2's and 3's has curling number 1, 2,  3, or $\ge 4$ (see \S\ref{Sec3a}).
The values of $c(n,1)$ are especially intriguing, as this is a 
combinatorial problem of independent interest.
The first 30 terms of $c(n,1)$ were contributed to \cite{OEIS} by
G. P. Srinivasan in 2006, who described it as the
``number of binary sequences of length $n$ with
no initial repeats'', which is equivalent to our definition
(see \seqnum{A122536}). 
However, we have been unable to find a formula
for $c(n,1)$\footnote{Apart from the conjectured asymptotic 
estimate \eqn{EqCn1a}.}, or even a recurrence that expresses $c(n,1)$ in terms of 
the values of $c(m,1)$ for $m<n$.
Theorem \ref{ThCK} says only that
\beql{Eqcn1}
c(n,1) ~=~ 2c(n-1,1) ~-~  [2 \mid n] ~ p'(n/2,1),
\eeq
$p'(n/2,1)$ being the number of robust primitive binary sequences
of length $n/2$ and curling number~1.

Use of \eqn{Eqcn1} enables $n$ terms of the $c(\cdot,1)$ sequence
to be obtained from $n/2$ terms of the $p'(\cdot,1)$ sequence. In practice, this limits us to
about 100 terms of the former sequence. In order to obtain more terms,
we introduce some further terminology (which will be used only
in this section).

If $S$ has length $n$, let $S^{[i]}$ denote its length-$i$ suffix, for $1 \le i < n$.
Then we define
\begin{align}
\sA(n,i)  & ~:=~ \{S \in \sC(n,1)  \mid \CN(S^{[i]}\,S)=1\},   & 1 \le i < n\,, \nonumber  \\
\sB(n,i)  & ~:=~ \{S \in \sC(n,1)  \mid \CN(S\,S^{[i]}\,S)=1\},   & 1 \le i < n\,, \nonumber  \\
\sE(n,i,j)  & ~:=~ \{S \in \sC(n,1)  \mid S^{[i]}\,S \in \sB(n+i,j)\},   & 1 \le i < n,
1 \le j <n+i\,,
\notag
\end{align}
and let $a(n,i) = \#\sA(n,i)$, $b(n,i) = \#\sB(n,i)$, 
$e(n,i,j) = \#\sE(n,i,j)$.
$S \succeq T$  will mean that $T$ is a suffix of $S$, and
$S \succ T$  that $T$ is a proper suffix of $S$.

The following two theorems give a canonical form (see \eqn{EqTh1}) for
 non-robust sequences with curling number 1.

\begin{thm}\label{ThA1}
If $\CN(S)=1$ but $\CN(TS)>1$ for some $T$ with $S \succ T$,
 then there exist $X \ne \EMPTY $, $Y \ne \EMPTY $ with
 \beql{EqTh1}
  S = XYX, \mbox{~where~}  \CN(X)=1, ~T \succeq Y, \mbox{~and~} X \succ Y\,.
 \eeq
\end{thm}
\begin{proof}
Since $\CN(TS) > 1$, $TS \succeq ZZ$ for some $Z$ with
$|Z| < |S|$, and therefore $S \succ Z$.
We write $S = XZ$ and observe that $TXZ \succeq ZZ$, so $TX \succeq Z$.
Therefore either $X \succeq Z$ or $Z \succ X$.
The former implies $\CN(S) > 1$, a contradiction.
So $Z \succ X$, say $Z = YX$, and $S=XYX$.

Since $S \succeq X$, $\CN(X)=1$. Also $TX \succeq Z = YX$, so $T \succeq Y$.
It remains to show that $X \succ Y$. Now $S \succ T \succeq Y$ and $S \succeq X$,
so either $Y \succeq X$ or $X \succ Y$.
The former implies $Y = WX$ for some $W$,
and then $S=XYX=XWXX$, contradicting $\CN(S)=1$. So $X \succ Y$.
\end{proof}

For example,  $S = 33223\,22333223$ has curling number 1 (the conspicuous
substring $333$ makes it easy to check this). If $T=2333\,223$,
then $TS = 2333\,Z^2$, where $Z  = 223\,33223$ (the spaces
in these strings are for legibility). Then, following the steps of the proof,
we write $S=XZ$, which defines $X = 33223$, and then write $Z=YX$,
which defines $Y=223$, and so finally we have
$$
S ~=~ XYX ~=~ 33223\,223\,33223\,,
$$
as claimed.

\begin{thm}\label{ThA2}
If $S = XYX = UVU$, with $X \succ Y \ne \EMPTY $, 
$U \succ V \ne \EMPTY $, $X \ne U$, then $\CN(S)>1$.
\end{thm}
\begin{proof}
Without loss of generality, $U \succ X$.
Since both $X$ and $U$ are {\em prefixes} of $S$, we have $U=XZ$ 
for some $Z \ne \EMPTY $, and $S \succ XZ$. 
Now $2|Z|=|S|-2|X|-|V| <|S|-2|X| = |Y| <|X|$,
so $|X|>|Z|$.
This implies $X \succ Z$ (they are both suffixes 
of $S$), say $X=AZ$, so
$S = UVU = UVXZ=UVAZZ$, contradicting $\CN(S)=1$. 
\end{proof}

Theorems \ref{ThA1} and \ref{ThA2} say that a non-robust sequence
$S$ with curling number 1 can be written in a unique way as $S = XYX$, where $Y$ is a suffix of $X$. 

\begin{cor}\label{CorA41}
(i) For $1 \le i < n/3$, there is a bijection between the sets $\sC(n,1) \setminus \sA(n,i)$ and
$$
\bigcup_{m = \lceil (n-i)/2 \rceil}^{\lfloor (n-1)/2 \rfloor} \, \sB(m, n-2m)\,.
$$
(ii) For $n/3 \le i < n$, there is a bijection between the sets $\sC(n,1) \setminus \sA(n,i)$ and
$$
\bigcup_{m = 1+\lfloor n/3 \rfloor}^{\lfloor (n-1)/2 \rfloor} \, \sB(m, n-2m)\,.
$$
\end{cor}
\begin{proof}
Fix $i$, where $1 \le i < n$.  First, suppose that $S$ is in $\sC(n,1)
\setminus \sA(n,i)$.  Taking $T = S^{[i]}$ in Theorem \ref{ThA1} we may
write $S = XYX$ where $m := |X| \ge 1$, $X \in \sC(m,1)$,
$n-2m = |Y| \le |S^{[i]}| = i$ and $m = |X| > |Y| = n-2m \ge 1$.
Hence $X \in \sB(m, n-2m)$ for some $m$ satisfying the three conditions:
$m \ge (n-i)/2$, $m > n/3$ and $m \le (n-1)/2$.  By Theorem \ref{ThA2}, $X$
belongs to only one such $\sB(m, n-2m)$.

Conversely, if $m$ satisfies these three conditions and $X \in
\sB(m, n-2m)$ then let $S = X X^{[n-2m]} X$.  By the definition
of $\sB(m, n-2m)$, $S$ must be in $\sC(n,1)$ and since $m \ge (n-i)/2$,
we have $S^{[i]} S \succeq S^{[n-2m]} S = (YX)^2$, so that S is not
in $\sA(n,i)$.

This establishes a bijection between $\sC(n,1) \setminus \sA(n,i)$
and the union of $\sB(m, n-2m)$ for $m$ satisfying the three earlier
conditions.  The proof is completed by observing that $(n-i)/2 > n/3$
if and only if $n > 3i$, which is the condition that separates  cases $(i)$ 
and $(ii)$ of the Corollary.
\end{proof}

Since the unions in Corollary \ref{CorA41} are clearly disjoint, we immediately
obtain the following formulas for $a(n,i)$.

\begin{cor}\label{CorA42}
(i) For $1 \le i < n/3$, 
\beql{EqA4a}
a(n,i) ~=~ c(n,1) ~-~ \sum_{m = \lceil (n-i)/2 \rceil}^{\lfloor (n-1)/2 \rfloor} \, b(m, n-2m)\,.
\eeq
(ii) For $n/3 \le i < n$,
\beql{EqA4b}
a(n,i) ~=~ c(n,1) ~-~ \sum_{m = 1 + \lfloor n/3 \rfloor}^{\lfloor (n-1)/2 \rfloor} \, b(m, n-2m)\,.
\eeq
\end{cor}

The next three theorems give a further refinement of
non-robust sequences, and lead to the set bijections and formulas in
Corollaries \ref{Cor13} and \ref{Cor14}. We postpone their proofs to the Appendix.
The proof that  Corollary \ref{Cor13} follows from Theorems \ref{ThA5} through \ref{ThA7}
is similar to the proof of Corollary \ref{CorA41} and is omitted.

\begin{thm}\label{ThA5}
If $X \succ Y$, $\CN(YX)=1$, and $\CN(XYX) > 1$,
then there exist $S$ and $T$ such that $YX = STS$ with $X \succeq T$,
$S \succ T$, $\CN(S)=1$.
Furthermore, either $|S|=|Y|$ or $|S| > 2|Y|$.
\end{thm}

\begin{thm}\label{ThA6}
If $n/2< i <n$ then there is a bijection between $\sA(n,i) \setminus \sB(n,i)$ and $\sB(i,n-i)$.
\end{thm}

\begin{thm}\label{ThA7}
If $1 \le i < n/3$ then $\sA(n,i) \setminus \sB(n,i)$ is a disjoint union
of $\sE(m-i, i, n+i-2m)$, where
$\max (2i, 1+\lfloor(n+i)/3 \rfloor ) < m \le \lfloor (n+i-1)/2 \rfloor$. 
\end{thm}

\begin{cor}\label{Cor13}
(i) For $1 \le i < n/5$, there is a bijection between $\sA(n,i) \setminus \sB(n,i)$ and
the disjoint union of  $\sE(m-i, i, n+i-2m)$, where 
$2i < m \le (n+i-1)/2 $.

(ii) For $n/5 \le i < n/3$, there is a bijection between $\sA(n,i) \setminus \sB(n,i)$ and
the disjoint union of  $\sE(m-i, i, n+i-2m)$, where 
$ (n+i)/3  < m \le (n+i-1)/2 $.

(iii) For $n/3 \le i \le n/2$, $\sB(n,i)$ is empty.

(iv) For $n/2 < i <n$, there is a bijection between $\sA(n,i) \setminus \sB(n,i)$
and $\sB(i,n-i)$. 
\end{cor}

\begin{cor}\label{Cor14}
(i)  For $1 \le i < n/3$,
\beql{Eq14a}
b(n,i) ~=~ a(n,i) ~-~ 
\sum_{m = \max ( 2i, 1+\lfloor (n+i)/3 \rfloor )}^{\lfloor (n+i-1)/2 \rfloor} e(m-i,i,n+i-2m)\,.
\eeq
(ii)
\beql{Eq14b}
b(n,i) ~=~ 0 \quad \mbox{~for~ } n/3 \le i \le n/2 \,.
\eeq 
(iii) 
\beql{Eq14c}
b(n,i) = a(n,i) - b(i,n-i) \quad \mbox{~for~ } n/2 < i <n \,.
\eeq
\end{cor}

Observing that $p'(n/2,1) = a(n/2,n/2-1)$,
equations \eqn{Eqcn1} and \eqn{EqA4a} through \eqn{Eq14c} can be used
recursively  to compute values of $c(n,1)$, using brute force to determine
$e(m,i,j)$ only for relatively small values of $m$ (see \S\ref{SecCg}).

We have briefly investigated the possibility of 
generalizing the approach in this section to
deal with curling numbers $k$ greater than one.
The following theorems replace
Theorems \ref{ThA1} and \ref{ThA2}:

\begin{thm}\label{ThSick2}
Suppose $S \in \sP(n,k) \setminus \sP'(n,k)$, where $k>1$.
Then there exist $X$ and $T$ with $S=X(TX)^k$ and $S \succ T$.
\end{thm}
\begin{proof}
By Theorem \ref{ThRobust}, $S^2 = PQ^{k+1}$ with $P \ne \EMPTY $.
If $(k+1)|Q| \ge n+|Q|-1$, then Theorem \ref{ThFW}
would imply that $S$ is periodic.
So $k|Q|<n-1$, and $k$ copies
of $Q$ lie properly inside $S$,
say $S=XQ^k$ with $|X|<|Q|$, $X \ne \EMPTY $.
Define $T$ by $Q=TX$ and we have $S=X(TX)^k$.
Also $PQ^{k+1}=SXQ^k$, so $PQ=PTX=SX$ and $S \succ T$.
\end{proof}

\begin{thm}\label{ThSick3}
The representation $S=X(TX)^k$ obtained in Theorem \ref{ThSick2}
is unique.
\end{thm}

Since we will not make any use of Theorem
\ref{ThSick3}, we omit the somewhat tedious proof.

Because $S \in \sP(n,k)$, we know that $S$ can be written as $XY^k$,
where $Y$ is primitive, possibly in several ways. Theorems \ref{ThSick2}
and \ref{ThSick3} say that if $S$ is not robust, then
exactly one of these $Y$'s has the corresponding $X$ as a suffix.
We have not pursued the generalizations of Theorems \ref{ThA5}--\ref{ThA7}
and Corollaries \ref{Cor13}--\ref{Cor14} to this case.

\subsection{ The values of $c(n,k)$ for $k \ge \lfloor \sqrt{n} \rfloor$ }\label{SecCe}

The second main theorem of this section gives an expression for $c(n,k)$ in the 
range $k \ge \lfloor \sqrt{n} \rfloor$
that involves the partial sum function $q(m,k)$.

\begin{thm}\label{ThSQRT}
We have $c(n,n)=2$ for all $n$, $c(n,n-1)=2$ for $n \ge 2$,
and, for $n \ge 4$ and $k \ge \lfloor \sqrt{n} \rfloor$,
\beql{EqTh1a}
c(n,k,\pi) ~=~
\begin{cases}
2^{n-(k+1)\pi} \, (2^{\pi}-1) \, q(\pi,k-1),
 & \mbox{if~} 1 \le \pi \le \lfloor \frac{n}{k+1} \rfloor \,, \\
2^{n-k\pi} \, q(\pi, k-1),
 & \mbox{if~} \lceil \frac{n+1}{k+1} \rceil \le \pi \le \lfloor \frac{n}{k} \rfloor \,, \\
\end{cases}
\eeq
and $c(n,k) =  \sum_{\pi=1}^{\lfloor n/k \rfloor} c(n,k,\pi)$.
\end{thm}

\begin{proof}
We assume $n \ge 4$ and $k \ge \lfloor \sqrt{n} \rfloor \ge 2$. 
Note that  $k \ge \lfloor \sqrt{n} \rfloor$
is equivalent to $k+1 > \sqrt{n}$. 

We consider the cases $n \le \pi(k+1)-1$ and $n \ge (k+1)\pi$ separately.

First, if $n \le \pi(k+1)-1$ , we have
$$
\left \lceil \frac{n+1}{k+1} \right \rceil \le \pi \le \left \lfloor \frac{n}{k} \right \rfloor\,.
$$
Let us write $S=X\,Y^k$, where $Y$ is minimal and has length $\pi$.
Then $n \le \pi(k+1)-1$ implies $|X| < \pi$.
By Lemma \ref{LemPi}, $Y \in \sQ(\pi,k-1)$.
There are $2^{n-\pi k}$ choices for $X$, and $q(\pi, k-1)$ choices for $Y$,
and we claim that the resulting sequence $X\,Y^k$ always has curling number $k$.
For suppose it has curling number $> k$, so that we have
$X\,Y^k = U\,V^{k+1}$, with $u=|U|$, $v=|V|$.
There are two sub-cases. If $(k+1)v \ge k\pi$, then we have $(k+1)\pi > |S| \ge (k+1)v$, implying $\pi > v$.
The two different representations of $S$ have a common suffix $Y^k$ of length $k\pi$,
which, since $k \ge 2$, satisfies
\beql{EqTh1d}
k\pi ~\ge~ v+\pi -1\,.
\eeq
By Theorem \ref{ThFW}, $Y=Z^g$, $V=Z^h$, with $g>h$, so $g\ge 2$, and $Y$ is imprimitive, a contradiction.
On the other hand, suppose $(k+1)v < k\pi$. Again $\pi > v$. Since $\CN(Y)<k$, $kv>\pi$.
Now the common suffix has length $(k+1)v$,  our inequalities imply
\beql{EqTh1c}
(k+1)v ~\ge~ v+\pi-1\,,
\eeq 
and, again by Theorem \ref{ThFW},  $Y$ is imprimitive, a contradiction.
So the number of sequences $S$ of this type is
$2^{n-k \pi} q(\pi,k-1)$, as claimed.

Second, if $n \ge (k+1)\pi$, we have 
$$
1 \le \pi \le \left \lfloor \frac{n}{k+1} \right\rfloor\,.
$$
Let us write
\beql{EqTh1b}
S ~=~ X B Y^k \,,
\eeq
where $X$ has length $n-(k+1)\pi$, $B$ has length $\pi$, and $Y \in \sQ(\pi, k-1)$.
Certainly $B \ne Y$ ($B$ stands for ``blocker'', the purpose of which is to
ensure that $Y$ is repeated only  $k$ times).
There are potentially $2^{n-(k+1)\pi}$ choices for $X$,
$2^\pi-1$ choices for $B$, and $q(\pi,k-1)$ choices for $Y$.
We claim that the assumption $k \ge \lfloor \sqrt{n} \rfloor$ guarantees that
all choices result in a sequence with curling number $k$.
For suppose on the contrary that $S$ (in \eqn{EqTh1b}) is also equal to
$U\,V^{k+1}$, with $u=|U|$, $v=|V|$. Again there are two sub-cases. 
If $(k+1)v \ge k\pi$, then we have
$$
(k+1)^2 ~>~ n ~\ge~ (k+1)v ~\ge~ k\pi\,,
$$
so $k+1>v$, $k \ge v$. The two different representations of $S$ have a common suffix
of length $k\pi$, and our inequalities imply \eqn{EqTh1d}.
On the other hand, suppose $(k+1)v < k\pi$. Again we have $kv > \pi $,  and the common suffix 
satisfies \eqn{EqTh1c}. In both cases Theorem \ref{ThFW} now leads to a contradiction.
This complete the proof of the theorem.
\end{proof}

The formulas in Theorem \ref{ThSQRT} cover a large
portion of the $c(n,k)$ table. However, although
with more work they could be
extended so as to apply to slightly smaller values of $k$,
it seems unlikely that this approach will lead to a formula for
$c(n,k,\pi)$ for small values of $k$.

\subsection{ The difference table $d(n,k)$ }\label{SecCf}

\begin{table}[thbp]
$$
\begin{array}{|r|rrrrrrrrrrrr|}
\hline
n\backslash k & 1 & 2 & 3 & 4 & 5 & 6 & 7 & 8 & 9 & 10 & 11 & 12 \\
\hline
2 &  2 &  -2 &    &    &    &    &    &    &    &    &    &    \\
3 &  0 &  2 &  -2 &    &    &    &    &    &    &    &    &    \\
4 &  2 &  -2 &  2 &  -2 &    &    &    &    &    &    &    &    \\
5 &  0 &  0 &  0 &  2 &  -2 &    &    &    &    &    &    &    \\
6 &  4 &  -2 &  -2 &  0 &  2 &  -2 &    &    &    &    &    &    \\
7 &  0 &  0 &  0 &  0 &  0 &  2 &  -2 &    &    &    &    &    \\
8 &  6 &  -6 &  2 &  -2 &  0 &  0 &  2 &  -2 &    &    &    &    \\
9 &  0 &  6 &  -6 &  0 &  0 &  0 &  0 &  2 &  -2 &    &    &    \\
10 &  10 &  -10 &  0 &  2 &  -2 &  0 &  0 &  0 &  2 &  -2 &    &    \\
11 &  0 &  0 &  0 &  0 &  0 &  0 &  0 &  0 &  0 &  2 &  -2 &    \\
12 &  20 &  -10 &  -4 &  -6 &  2 &  -2 &  0 &  0 &  0 &  0 &  2 &  -2 \\
\hline
\end{array}
$$
\caption{The difference table $d(n,k)$ defined by \eqn{EqDnk}
(for an extended table see \seqnum{A217943}).}
\label{TabDnk}
\end{table}

In the $c(n,k)$ table (Table \ref{Tabcnk}), if we look at
the difference between each row and twice the previous row,
we obtain a much simpler table.\footnote{It was by studying
the $d(n,k)$ table that we were led to Theorems \ref{ThCK} and \ref{ThSQRT}.}
We define
\beql{EqDnk}
d(n,k) ~:=~ 2\,c(n-1,k) ~-~ c(n,k) \,,
\eeq
for $n \ge 2$, $1 \le k \le n-1$, with $d(n,n)=-2$.
The initial values are shown in Table \ref{TabDnk}.
We see that if one ignores
the initial entries in each row,
most of the remaining entries are zero,
except for diagonal lines of pairs of nonzero entries.
More precisely, it appears that
\begin{align}\label{EqDconj}
d(2k,k-1)  & ~=~ -d(2k,k) ~=~ 2, \quad k \ge 4 \,, \nonumber \\
d(3k,k-1)  & ~=~ -d(3k,k) ~=~ 6, \quad k \ge 5 \,, \nonumber \\
d(4k,k-1)  & ~=~ -d(4k,k) ~=~ 12, \quad k \ge 6 \,, \nonumber \\
d(5k,k-1)  & ~=~ -d(5k,k) ~=~ 30, \quad k \ge 7 \,, \nonumber \\
\end{align}
and so on.
Only the first of these diagonal lines can be seen in Table \ref{TabDnk},
but they are all visible in the extended table that is given in 
entry \seqnum{A217943} in \cite{OEIS}. 
These expressions all follow from Theorem \ref{ThSQRT}:

\begin{thm}\label{Thdnk}
In the range $k \ge \lfloor \sqrt{n} \rfloor$, the only nonzero entries
in the $d(n,k)$ table are
\beql{EqDconj2}
d(mk,k-1) ~=~ -d(mk,k) ~=~ q(m,m), \mbox{~for~} m \ge 1, ~k \ge m+2\,.
\eeq
\end{thm}
\begin{proof}
This follows easily from Theorem \ref{ThSQRT}. We prove
the second assertion in \eqn{EqDconj2} as an illustration.
We have
\beql{Eq100}
d(mk,k) ~=~ 2c(mk-1,k)~-~c(mk,k) \,.
\eeq
  From \eqn{EqTh1a},
\beql{Eq101}
c(mk,k)~=~ \sum_{\pi=1}^{m-1} \,c(mk,k,\pi) ~+~ c(mk,k,m) \,,
\eeq
\beql{Eq102}
c(mk-1,k)~=~ \sum_{\pi=1}^{m-1} \,c(mk-1,k,\pi)\,.
\eeq
Each summand in \eqn{Eq101} (see \eqn{EqTh1a}) is exactly twice
the corresponding term in \eqn{Eq102}, and $c(mk,k,m)$ $= q(m,k) = q(m,m)$, 
so $d(mk,k) = - q(m,m)$.
\end{proof}

Note that whereas the expression for $c(n,k)$ in Theorem \ref{ThSQRT}
involves the general function $q(\pi,k)$, the expression for $d(n,k)$
in the range $k \ge \lfloor \sqrt{n} \rfloor$ is fully explicit, since $q(m,m)$
is given by \eqn{Eqalpha}.

Theorem \ref{ThCK} gives another formula for $d(n,k)$:
\beql{EqDPQ}
d(n,k) ~=~ [k+1 \mid n] 
\left( p'\left(\frac{n}{k+1},k\right) + q\left( \frac{n}{k+1},k-1 \right) \right)
~-~ [k \mid n] \left( p'\left(\frac{n}{k},k-1\right) + q\left( \frac{n}{k},k-2 \right) \right)\,,
\eeq
and in particular,
\begin{align}
        d(n,1) & ~=~ [2 \mid n]~ p'(n/2,1) \,, \nonumber  \\
        d(n,2) & ~=~ [3 \mid n] ~ (p'(n/3,2) + p(n/3,1)) ~-~ [2 \mid n] ~p'(n/2,1)\,.
\end{align}
The first of these is nicely checked by noticing that the nonzero entries
in the first column of the $d(n,k)$ table, namely $2,2,4,6,10,20,\cdots$ are also
the entries in the first column of the $p'(n,k)$ table (Table \ref{Tabppnk}).
It is also worth mentioning that if $p$ is prime then $c(p,k)=2c(p-1,k)$ for
all $k$ (see \eqn{Eqcn1}) and so $d(p,k)=0$.

\subsection{ Computation of $c(n,k)$}\label{SecCg}

We constructed an extensive table of values of $c(n,k)$,
hoping that it would lead to additional insight into these numbers.
First, by direct enumeration, using a number of different
programs and different computers
(including a four-day computation on a cluster of 64 SPARC processors), 
we calculated $c(n,k)$ for $n \le 51$.

Second, we tabulated $e(n,i,j)$ for $n \le 23$.
This was sufficient for
the recurrences \eqn{Eqcn1} and \eqn{EqA4a}--\eqn{Eq14c}
to give $c(n,1)$ for $n \le 200$.
These values suggest the conjecture that
\beql{EqCn1a}
\lim_{n \rightarrow \infty} ~ \frac{ c(n,1)}{2^n} ~=~ 0.27004339525895354325\cdots \,.
\eeq
  From Equation \eqn{Eqcn1} we have
$$
c(n,1) ~\ge~ 2c(n-1,1)~-~ [2 \mid n] \,c(n/2,1)\,,$$
which implies, using the known values of $c(n,1)$, that
\beql{EqCn1b}
c(n,1) ~>~ 0.27 \cdot 2^n \quad \mbox{~for~} n \ge 200\,.
\eeq
We omit the proof. But we have no comparable upper bound
for $c(n,1)$ (other than $2^n$),
nor a proof that the limit \eqn{EqCn1a} exists.

Third, we used a different approach, which enabled us to take a table of
the curling numbers of all sequences of length $n \le n_0$,
and from this produce a table of $c(n,k)$ for all $n \le 2n_0$,
without having to compute the curling numbers
of all $2^{2n_0}$ sequences of length $2n_0$.
The idea underlying this approach is the following.
Consider a sequence $S$ of length $n$ with
$n_0 \le n \le 2n_0$, and let $M$ be its length-$n_0$ suffix.
As a first approximation, we set $\CN(S) = \CN(M) = l$ (say).
This approximation will be wrong if for some suffix $T$ of $M$
it should happen that $T^{l+1}$ is a suffix of $S$.
If so, we must increase $\CN(S)$ by 1 for all $S$ having
suffix $T^{l+1}$. There are complications
if there is more than one such $T$ to be considered,
but the Fine-Wilf theorem (Theorem  \ref{ThFW}) shows that this 
can only happen when $l=1$. We omit discussion of the details.
Using this approach (with $n_0 = 32$) we were able to extend the table of values
of $c(n,k)$ and $p(n,k)$ to $n=64$. 

Finally, we tabulated $p'(n,k)$ for $n \le 36$.
This, together with the 200 terms of $c(n,1)$, was sufficient
for the recurrence in Theorem \ref{ThCK} to give
the first 104 rows of the $c(n,k)$ table. These results can be seen in 
\seqnum{A216955} and \seqnum{A122536}.

\section{ Tail lengths of \{2,3\}-sequences }\label{Sec4}

\subsection{ Distribution of tail lengths}\label{Sec3t}

Let $t(n,i)$ denote the number of starting sequences $S_0$ of $n$
2's and 3's which have tail length $i$,
where $i$ ranges from 0 to $\TAIL(n)$.
The initial values are shown in Table \ref{TabTnk}.
Since the rows rapidly increase in length (cf. Table \ref{Tabmu}),
we end this table at $n=9$.
Note that the entries for $i=9$ through 55 (which are all zero) 
have been compressed into a single column.
Rows $n=22$ and 32 are shown in Tables \ref{TabTail22} and \ref{TabTail32}.
Entry \seqnum{A217209} in \cite{OEIS} gives the first 48 rows in full.
The first column is the same as the first column of the $c(n,k)$ table,
and contains the numbers $c(n,1)$ that are the subject of \S\ref{SecCd}.

\begin{table}[!h]
$$
\begin{array}{|r|cccccccccccccc|}
\hline
n\backslash i & 0 & 1 & 2 & 3 & 4 & 5 & 6 & 7 & 8 & 9\mbox{-}55 & 56 & 57 & 58 & 59 \\
\hline
1 & 2 &   &   &   &   &   &   &   &   &    &    & & &   \\
2 & 2 & 1 & 1 &   &   &   &   &   &   &    &    & & &   \\
3 & 4 & 2 & 2 &   &   &   &   &   &   &    &    & & &   \\
4 & 6 & 5 & 3 & 1 & 1 &   &   &   &   &    &    & & &   \\
5 & 12 & 9 & 6 & 2 & 3 &   &   &   &   &    &    & & &   \\
6 & 20 & 18 & 12 &  6 &  7 & 0 & 0 & 0 & 1 &    &    & & &   \\
7 & 40 & 34 & 25 & 11 & 14 & 1 & 0 & 1 & 2 &    &    & & &   \\
8 & 74 & 71 & 47 & 24 & 28 & 1 & 3 & 2 & 3 & 0  & 0 & 2 & 1 &   \\
9 & 148 & 139 & 95 & 48 & 56 & 6 & 4 & 3 & 6 & 0  & 2 & 3 & 1 & 1 \\
\hline
\end{array}
$$
\caption{Table of $t(n,i)$, the number of sequences of $n$ 2's and 3's
with tail length $i$, for $0 \le i \le \TAIL(n)$ (\seqnum{A217209}).}
\label{TabTnk}
\end{table}

\begin{table}[!h]
{\scriptsize
$$
\begin{array}{r|rrrrrrrrr}
0\mbox{-}8 &  1133200 &  1140102 &  768386 &  417081 &  479224 &  47190 &  33440 &  32283 &  51035 \\
9\mbox{-}17 &  6388 &  6096 &  1031 &  2074 &  516 &  807 &  67 &  0 &  0 \\
18\mbox{-}26 &  1 &  1 &  3 &  6 &  12 &  7 &  0 &  0 &  0 \\
27\mbox{-}35 &  0 &  0 &  0 &  0 &  0 &  0 &  0 &  0 &  0 \\
36\mbox{-}44 &  0 &  0 &  0 &  0 &  0 &  0 &  1 &  7 &  16 \\
45\mbox{-}53 &  24 &  50 &  98 &  198 &  394 &  786 &  1316 &  2633 &  5121 \\
54\mbox{-}62 &  5891 &  7687 &  9230 &  14622 &  12983 &  6486 &  2659 &  642 &  1099 \\
63\mbox{-}71 &  463 &  299 &  32 &  0 &  0 &  0 &  0 &  0 &  0 \\
72\mbox{-}80 &  1 &  0 &  0 &  0 &  0 &  0 &  0 &  0 &  0 \\
81\mbox{-}89 &  0 &  0 &  0 &  0 &  0 &  0 &  0 &  0 &  0 \\
90\mbox{-}98 &  0 &  0 &  0 &  0 &  0 &  0 &  0 &  0 &  0 \\
99\mbox{-}107 &  0 &  0 &  1 &  1 &  2 &  4 &  8 &  16 &  32 \\
108\mbox{-}116 &  64 &  128 &  256 &  512 &  1024 &  17 &  17 &  34 &  70 \\
117\mbox{-}120 &  139 &  282 &  8 &  1 \\
\end{array}
$$
}
\caption{ Distribution of tail lengths $t(22,i)$, $0 \le i \le 120$, for all starting sequences of length 22 (22 is the first time a tail of length 120 is reached). Note the three ``clumps.'' }
\label{TabTail22}
\end{table}

\begin{table}[!h]
{\scriptsize
$$
\begin{array}{r|rrrrrrrrr}
0\mbox{-}8 &  1159845258 &  1167273283 &  786757853 &  427198253 &  490970976 &  48399112 &  34266983 &  33065461 &  52260747 \\
9\mbox{-}17 &  6585936 &  6286710 &  1088875 &  2157877 &  553922 &  848516 &  69469 &  519 &  1038 \\
18\mbox{-}26 &  836 &  1547 &  3092 &  6184 &  11843 &  7303 &  206 &  28 &  57 \\
27\mbox{-}35 &  99 &  194 &  0 &  0 &  0 &  2 &  9 &  21 &  34 \\
36\mbox{-}44 &  72 &  130 &  198 &  394 &  788 &  1576 &  3153 &  6305 &  12610 \\
45\mbox{-}53 &  25219 &  50438 &  100876 &  201752 &  403504 &  804960 &  1347868 &  2695736 &  5244019 \\
54\mbox{-}62 &  6034490 &  7874728 &  9455010 &  14977616 &  13308516 &  6658834 &  2742615 &  676305 &  1153446 \\ 
63\mbox{-}71 &  487704 &  309650 &  32814 &  28 &  24 &  48 &  96 &  193 &  385 \\
72\mbox{-}80 &  770 &  0 &  0 &  0 &  0 &  0 &  0 &  0 &  0 \\
81\mbox{-}89 &  0 &  0 &  0 &  0 &  1 &  2 &  5 &  10 &  20 \\
90\mbox{-}98 &  0 &  1 &  1 &  2 &  4 &  8 &  16 &  32 &  64 \\
99\mbox{-}107 &  128 &  256 &  512 &  1024 &  2048 &  4096 &  8192 &  16384 &  32768 \\
108\mbox{-}116 &  65536 &  131072 &  262144 &  524288 &  1048544 &  18331 &  18265 &  36530 &  73119 \\
117\mbox{-}120 &  146237 &  292601 &  8798 &  1144 &             &         &       &        &        \\
\end{array}
$$
}
\caption{ Distribution of tail lengths $t(32,i)$, $0 \le i \le 120$, for
all starting sequences of length 32. The clumps have thickened. }
\label{TabTail32}
\end{table}

As can be seen from Tables \ref{TabTnk}--\ref{TabTail32},
the values in each row are distributed into clumps, with each clump
gradually thickening as $n$ increases.
Table \ref{TabTail22} shows the distribution of tail lengths 
at length 22, the first time that a tail of length 120 is reached
(note the final ``1'', indicating that the starting sequence was unique).
By length 32 (Table \ref{TabTail32}), the clumps have thickened
but still end at 120. A tail of length greater than 120 does not appear until length 48,
when the greatest tail length jumps to 131.
The powers of 2 in Tables \ref{TabTail22} and \ref{TabTail32}
suggest that the clumps tend to grow by
prefixing good starting sequences of shorter length by
random strings of 2's and 3's. 
However, we do not have a satisfactory model which explains this distribution.

The mean value of the $n$th row,
$$
\frac{1}{2^n} ~ \sum_{i=0}^{\TAIL(n)} i \, t(n,i) \,,
$$
at least for $n \le 48$, is converging 
to a value around $2.741\cdots$ (see \seqnum{A216813}). 
That is, if a starting sequence consisting of $n$ 2's and 3's
is chosen at random, it will reach a 1 on average after
only $2.741\cdots$ steps.
This is in sharp contrast to the behavior of the best 
starting sequences, as we see from Table \ref{Tabmu}.
Of course if the curling number conjecture is false for
sequences of 2's and 3's, the mean will be
infinite beyond some point.

\subsection{ A probabilistic model }\label{Sec3a}
Let $\PROB _k^{(n)} := c(n,k)/2^n$ denote the probability that a randomly chosen
sequence consisting of $n$ 2's and 3's has curling number $k$.
The available data ($n \le 200$ for $k=1$, $n \le 104$ for $k>1$)
suggests that as $n$ increases these probabilities are converging to the values
$$
\PROB _1 ~\approx~ .270,
~ \PROB _2 ~\approx~ .434,
~ \PROB _3 ~\approx~ .162,
~ \sum_{k \ge 4} \PROB _k ~\approx~ .134 \,.
$$
When we extend a sequence $S$ by appending the curling number $k= \CN(S)$,
{\em if} it were the case that the concatenation $S k$ were
independent of $S$, we could model this
process as a two-state Markov chain with states
``curling number is 2 or 3'' and ``curling number is 1 or $\ge 4$.''
The probability of staying in the ``2 or 3'' state would be
$\PROB _2 + \PROB _3 \approx .596\cdots$ and the probability of
leaving that state would be $.404\cdots$.
If the starting sequence is randomly chosen
from all $2^n$ possibilities, this model would imply that
the maximal number of steps before reaching the ``1 or $\ge 4$''
state for the first time would be about
$$
t ~\approx~ n \, \frac{\log 2}{\log (1/.596)} ~\approx~ 1.34 \, n \,. 
$$
This Markov model certainly does not apply at the
beginning of the appending process, but it could conceivably 
be valid once the sequence has been extended for a while,
so we think it is worth mentioning.

\subsection{ ``Rotten'' sequences: prefix decreases tail }\label{Sec3r}
Let $S_0$ be an arbitrary sequence of 2's and
3's of length $n$,
with tail length $\tau(S_0) = i$, say.
It seems plausible that if $n$ is large, then prefixing
$S_0$ by a single 2 or 3 will not change $\tau (S_0)$,
i.e., that $\tau(2 \, S_0) = \tau(3 \, S_0) = \tau(S_0)$. 
But could doing this actually {\em decrease} the tail length?
Choosing an adjective not normally used in mathematics,
we will call $S_0$ {\em rotten} if either
 $\tau(2 \, S_0) < \tau(S_0)$ or
 $\tau(3 \, S_0) < \tau(S_0)$, and {\em doubly rotten}
if both $\tau(2 \, S_0) < \tau(S_0)$ and
 $\tau(3 \, S_0) <  \tau(S_0)$ hold.
There are surprisingly few rotten sequences of length up
through 34. The first few examples are shown in Table \ref{TabRott},
and the numbers of rotten sequences of lengths 1 through 34 are
given in Table \ref{TabRott2}.
If $S_0 = 3\,2\,3\,2\,3$, for example,
then $S_0^{(e)} = 3\,2\,3\,2\,3\,2\,3\,3\,2$,
and $\tau(S_0)=4$.
But if we prefix $S_0$ with a 2, so the starting sequence is $2\,S_0 = 2\,3\,2\,3\,2\,3$,
the extension is $2\,3\,2\,3\,2\,3\,3\,2$, so $\tau(2\,S_0) = 2$, and $S_0$ is rotten.

\begin{table}[!h]
{\scriptsize
$$
\begin{array}{rrrrrrrrrr}
22 & 333 & 32323 & 323232 & 2323232 & 3232323 & 22322232 \\
23222322 & 23223223 & 33233233 & 223222322 & 223222323 & 232223222 & 332332332 \\
2232223222 & 2232223223 & 2232223232 & 2322232223 & 2322322322 & 2332332332 & 3322332233 \\
3323323323 & 22322232223 & 22322232232 & 22322232322 & 22322322232 & 22322322322 & 22323222322 
\end{array}
$$
}
\caption{The first 28 rotten sequences (\seqnum{A216730}).}
\label{TabRott}
\end{table}

\begin{table}[!h]
$$
\begin{array}{rrrrrrrrrr}
0 & 1 & 1 & 0 & 1 & 1 & 2 & 4 & 4 & 8 \\
14 & 11 & 18 & 30 & 26 & 24 & 40 & 35 & 58 & 69 \\
48 & 84 & 158 & 67 & 139 & 287 & 215 & 242 & 490 & 323 \\
624 & 919 & 516 & 1072
\end{array}
$$
\caption{Number of rotten sequences of lengths 1 through 34 (\seqnum{A216950}).}
\label{TabRott2}
\end{table}

However, up to length 34 there are no doubly rotten sequences. 
\begin{conj}\label{ConjDR}
Doubly rotten sequences do not exist.
\end{conj}

If this conjecture were true, it would imply
that one can always prefix a starting sequence $S_0$
by one of $\{2,\,3\}$ without decreasing the tail length. 
This would explain the observation made in \S\ref{Sec2P}
about the behavior of $\TAIL(n)$ between jump points.
It would also imply that $\TAIL(n+1) \ge \TAIL(n)$ for 
all $n$, something that we do not know at present.

\subsection{ Sequences in which first term is essential}\label{Sec3b}
A statistic that is relevant to the study
of rotten sequences is the following.
If a starting sequence $S_0$ of length $n$ is chosen at random, 
and has curling number $k$, this means we can write
$S_0 = X Y^k$ for suitable sequences $X, Y$.
What is the probability that we must necessarily
take $X$ to be the empty sequence, i.e., that
the only such representation goes all the way back
to the beginning of $S_0$
(and so the first term is essential for the computation
of the curling number)? The sequence
$2\,2\,3\,2\,2\,3$ is an example,
since here $k=2$, and $X=\EMPTY $, $Y=2\,2\,3$ is the only
representation. But 
$2\,3\,3\,2\,3\,3$ is not, since $k=2$ and we can either
take $X=\EMPTY $, $Y=2\,3\,3$ or
$X=2\,3\,3\,2$, $Y=3$, and the latter representation avoids using $X=\EMPTY $. 
The number of such sequences of length $n$ for $1 \le n \le 35$ is
given in Table \ref{TabRott3}. 
If $n$ is prime, the number is 2, but the
limit supremum of these numbers
appears to grow exponentially.

\begin{table}[!h]
$$
\begin{array}{rrrrrrrrrr}
2 & 2 & 2 & 4 & 2 & 8 & 2 & 10 & 8 & 14 \\
2 & 40 & 2 & 40 & 32 & 88 & 2 & 192 & 2 & 324 \\
100 & 564 & 2 & 1356 & 32 & 2226 & 370 & 4564 & 2 & 9656 \\
2 & 17944 & 1450 & 35424 & 152
\end{array}
$$
\caption{Number of sequences of lengths 1 through 35
whose curling number representation $X\,Y^k$ requires $X=\EMPTY $ (\seqnum{A216951}).}
\label{TabRott3}
\end{table}

\subsection{ Sequences where prefix increases tail}\label{Sec3p}
In contrast to ``rotten'' sequences, we also investigated
starting sequences $S_0$ for which either 
$\tau(2\, S_0) > \tau(S_0)$ or $\tau(3\, S_0) > \tau(S_0)$.
The sequence $S_0 = 2\,2\,3\,2\,2$ is an example, since
$\tau(S_0)=2$, $\tau(2\,S_0)=8$, $\tau(3\,S_0)=2$.
The numbers of such sequences of lengths 1 through 30 are shown in
Table \ref{TabRipe}. There are rather more
of these than there are rotten sequences, although we found
no example where both 
$\tau(2\, S_0) > \tau(S_0)$ and $\tau(3\, S_0) > \tau(S_0)$
hold.

\begin{table}[!h]
$$
\begin{array}{rrrrrrrrrr}
2 &  1 &  2 &  1 &  5 &  3 &  12 &  9 &  19 &  16 \\
38 &  20 &  59 &  42 &  104 &  65 &  213 &  111 &  400 &  245 \\
765 &  439 &  1563 &  820 &  3046 &  1731 &  5955 &  3292 &  12078 &  6343 \\
\end{array}
$$
\caption{Number of sequences $S_0$ of lengths 1 through 30 such that
$\tau(2\, S_0) > \tau(S_0)$ or $\tau(3\, S_0) > \tau(S_0)$ (\seqnum{A217437}).}
\label{TabRipe}
\end{table}

\section{Gijswijt's sequence}\label{Sec3}

If we simply start with $S_0 = 1$, and generate an infinite sequence
by continually appending the curling number of the current
sequence, as in \eqn{Eq1}, we obtain
$$
G :=  1\,1\,2\,1\,1\,2\,2\,2\,3\,1\,1\,2\,1\,1\,2\,2\,2\,3\,2\,1\,1\,2\,1\,1\,2\,2\,2\,3\,1\,1\,2\,1\,1\cdots \,.
$$
This is {\em Gijswijt's sequence}, \seqnum{A090822},
invented by D. Gijswijt in 2004, and analyzed by
van de Bult et al.  \cite{GIJ}.

The first time a $4$ appears in $G$ is at term $220$. One can calculate
quite a few million terms without finding a $5$ (as
the authors of \cite{GIJ} discovered), 
but in \cite{GIJ} it was shown that a $5$ 
eventually appears for the first time at about term
$$
10^{10^{\Snj 23}} \, .
$$
Van de Bult et al. \cite{GIJ} also show that $G$ is in fact unbounded, 
and conjecture that the first time that
a number $m~\ge 6$ appears is at about term number
$$
2^{2^{\Snj 3^{\Snj 4^{\cdot^{\cdot^{\cdot^{\Snj {m-1}}}}}}}} \,,
$$
a tower of height $m-1$.
The fairly complicated arguments used in  \cite{GIJ} could be considerably simplified
and extended if the curling number conjecture were known to be true. 

Our final theorem shows that if the curling number conjecture is true,
any starting sequence $S$ that
does not contain a 1 must eventually merge with $G$.

\begin{thm}\label{ThJoin}
Assume the curling number conjecture is true.
Let $S$ be an initial sequence not containing a $1$,
let $S^{(e)}$ be its ``extension'' $($defined in \S\ref{Sec1}$)$,
and let $S^{(\infty)}$ be its infinite continuation.
Then $S^{(\infty)} = S^{(e)}G$.
\end{thm}
\begin{proof} 
By definition, $S^{(e)}$ does not contain a 1 but is immediately followed by a 1.
Suppose  $S^{(\infty)} \ne S^{(e)} G$, and suppose they first differ at a position
where $S^{(\infty)}$ is $n$, say, whereas $S^{(e)} G$ is $m<n$.
This $n$ must be the curling number of some portion of  $S^{(\infty)}$ that 
begins with a suffix $X$, say, of $S^{(e)}$. Let $S^{(e)} = W X$.
Then $S^{(\infty)} = W (X T)^n \,n\cdots$ for some prefix $T$ of $G$,
whereas $G = T (X T)^{n-1}\,m\,\cdots$\,.
If $n=2$, $m=1$, we have $G=TXT1\cdots$\,.
The curling number of the first copy of $T$ is the first term of $X$, which is not 1,
but the curling number of the second $T$ is 1, a contradiction.
On the other hand, if $n\ge 3$, $G = TXTXT\cdots XTm\cdots$\,,
and the initial $TXTX$ has curling number at least 2 and cannot be followed by 
$T$ (which begins with 1), again a contradiction.
\end{proof}

We do not know if the theorem is still true if $S$ is allowed to contain
a 1 but does not end with 1.

\section{ Open questions and topics for future research}
\begin{enumerate}
\item
Is the curling number conjecture (even just for the case of sequences of
2's and 3's) true?
\item
It would be nice to have some further exact values of $\TAIL(n)$,
beyond $n = \LIMA$, even though they will require extensive computations.
\item
What is the asymptotic behavior of $\TAIL(n)$?
\item
Can the especially good starting sequences shown in Tables
\ref{TabGood1} and \ref{TabGood2}
(in particular those of lengths 22, 48 and 77) be generalized?  
What makes them so special?
\item
Can the properties of good starting sequences mentioned in 
Conjecture \ref{Conj33} be justified?
\item
Can Shirshov's theorem (see \S\ref{Sec2S}) be
modified so as to apply to our problem?
\item
Are there analogs of Theorems \ref{ThCK} and \ref{ThSQRT} for
$p(n,k)$ (the number of primitive sequences) or 
$p'(n,k)$ (the number of primitive and robust sequences) ?
\item
Are there formulas for
$c(n,k)$ that are more explicit than those given in Theorems \ref{ThCK} 
and \ref{ThSQRT}?
Is there a formula that matches the 200 known terms of the
$c(n,1)$ sequence?
\item
Are there formulas or recurrences 
for the numbers $t(n,i)$ of starting sequences
with tail length $i$?
\item
Is there a probabilistic model that better explains the distribution
of values of $t(n,i)$ visible in Tables \ref{TabTnk}--\ref{TabTail32}
and \seqnum{A217209}?
The model presented in \S\ref{Sec3a} is certainly inadequate.
\item
Do ``doubly rotten'' sequence exist? (See Conjecture \ref{ConjDR}.)
\item
The question implicit in the last sentence of \S\ref{Sec3}.
\end{enumerate}

\section{Appendix: Proofs of Theorems \ref{ThA5}, \ref{ThA6}, \ref{ThA7}}

\subsection{ Theorem \ref{ThA5} }
\begin{proof}
The first statement follows immediately
from Theorem \ref{ThA1}, taking $S$ and $T$ in that theorem to
be $YX$ and $X$ respectively.
To prove the second statement, let
$x  := |X|$, $y  := |Y|$, $s  := |S|$, $t  := |T|$,
and note that $YX = STS$ implies $x + y = 2 s + t$.
Also $X=BY$ (say), with $B \ne \EMPTY $.

We are to show that $s = y$ or $s > 2 y$.  
First, suppose that $y < s \le 2 y$.
Since $s > y$, there exists 
a sequence $U$ with $|U| = s-y$
such that $S = YU$.  Then
we have the following chains of implications
[the successive assertions are enclosed in
square brackets]:
$
[ s > y  ]
 \Rightarrow [s > y - t ]
 \Rightarrow [x = 2 s + t - y > s ]
 \Rightarrow [X \succ S \succ U ],
$
and
$
 [ s \le 2 y ]
\Rightarrow [ s - y \le y ]
\Rightarrow [ |Y| \le |U| ]
\Rightarrow [ Y \succeq U ~(\mbox{since~} YX=YBY=STYU)]
\Rightarrow [ Y=CU ~(\mbox{say}) ]
\Rightarrow [ X \succ S=CUU]
\Rightarrow [ \CN(X) > 1 ],
$
a contradiction.

Second, suppose that $s < y$.  Then there exists $U \ne \EMPTY $ 
with $|U| = y-s$ such that $Y = SU$.  
But $x>y>s$ and $YX=STS$ imply $X \succ S$, 
and since $X \succ Y$ then $X \succ U$ also.  
If $S \succeq U$ then $YX = YBSUX \succ UU$, which contradicts $\CN(YX) = 1$. 
Hence
$
[ U \succ S ]
        \Rightarrow [ s < y - s ]
        \Rightarrow [ 2 s < y ]
        \Rightarrow [ x + y = 2 s + t < y + t \le y + x ]\,,
$
since $X \succeq T$. Since this is impossible, 
$s < y$ is also impossible.
\end{proof}

Note that the condition $|S| = |Y|$ is equivalent to $2 |Y| > |X|$:
if $s=y$ then $x = y + t$, which implies $2y = s+y >t+y$ (since $t>s$).
Conversely, if $s \ne y$ then $s>2y$, which implies 
$x+y=2s+t>4y+t$, $x>3y+t$, so $x \ge 2y$.
Similar reasoning shows that 
the condition $|S| > 2 |Y|$ is equivalent to $3 |Y| < |X|$.

\subsection{ Theorem \ref{ThA6} }
\begin{proof}
If $X \in \sA(n,i) \setminus \sB(n,i)$ then we may apply Theorem \ref{ThA5}
to $X$, taking $Y = X^{[i]}$, with $|X|=n$, $|Y|=i$, where $n/2 <i <n$.
So there exist $S$, $T$ with  $YX = STS$,  $Y \succeq T$,
$S \succ T$, and either $|S|=|Y|$ or $|S|>2|Y|$. 
We cannot have $|S|>2|Y|$, since that implies $|S|>n$, $2|S|>2n>|YX|$,
which contradicts $YX=STS$.
So $|S|=|Y|$, $Y=S$, $X=TY$, and $|T|=n-i$.
Also $\CN(YX)=1$ by definition of $\sA(n,i)$,
i.e., $\CN(YTY)=1$, so $Y \in \sB(i,n-i)$.

The map from $X$ to $Y$ is one-to-one, since $X$ determines $Y$.
To show it is onto, take $Y \in \sB(i,n-i)$, let $Q=Y^{[n-i]}$, and define
$P$ by $Y=PQ$ and set $X := QY = QPQ$.
Then we have $\CN(YQY) = \CN(YX)=1$, so $X \in \sA(n,i)$.
Also $XYX=QPQ\,PQ\,QPQ$ has curling number at least 2, so $X \notin \sB(n,i)$.  
Hence $X \in \sA(n,i) \setminus \sB(n,i)$.
\end{proof}

\subsection{ Theorem \ref{ThA7}}
\begin{proof}
Since the sets $\sE$ in the sum are clearly
disjoint, we just need to establish a bijection between the elements of
$\sA(n,i) \setminus \sB(n,i)$ and the disjoint union of the $\sE$ sets defined by the range
of $m$.

As in the previous proof, if
$X \in \sA(n,i) \setminus \sB(n,i)$, then we may apply Theorem \ref{ThA5}
to $X$, taking $Y = X^{[i]}$, with $|X|=n$, $|Y|=i$, where now $1 \le i < n/3$.
There exist $S$, $T$ with  $YX = STS$,  $Y \succeq T$,
$S \succ T$, and either $|S|=|Y|$ or $|S|>2|Y|$. 
Let $|S| = m$, $|T| = n+i-2 m$. As before, $S \in \sB(m,|T|)$.  There are three
conditions that $m$ must satisfy:
(i) $|T| \ge 1$  implies $m \le (n+i-1)/2$;
(ii) $|S| > |T|$ implies $m > \lceil (n+i)/3 \rceil$;
(iii)  $m=i<n/3$ is incompatible with $YX=STS$, so $m > 2i$.

Since $m > i$, we
may write $S = YU$, with $|U| = m-i$.  
Since $m > 2 i$, $m-i > i$ and $|U|>|Y|$.
Now $X \succ S$, so $X \succ U$ and therefore $U \succ Y$.  
Since $S = YU \in \sB(m,|T|)$,
$U \in \sE(m-i,i,|T|)$.  
The mapping $X \mapsto U$ is one-to-one
since $X$ determines $Y = X^{[i]}$, 
$S$ and $T$ are unique by Theorem \ref{ThA2},
$m=|S|$, and $S=YU$ determines $U$.

To show the map is onto, suppose $U \in \sE(m-i,i,n+i-2 m)$ for some $m$ satisfying
conditions (i)-(iii) above.
Then set $Y = U^{[i]}$, $S = YU$, $T = S^{[n+i-2 m]}$,
and $X = UTS$.  
Then $YX = STS$ so that $U \in \sE(m-i,i,n+i-2 m)$ implies 
$YX \in \sA(n,i)$.  
But $XYX = XSTS \succ TSTS$,
so $\CN(XYX) > 1$  and therefore $XYX \notin \sB(n,i)$.
\end{proof}

\section{Acknowledgments}
We thank the referees for several helpful comments.

\bigskip
\hrule
\bigskip

\noindent 2010 {\it Mathematics Subject Classification}:
Primary 68R15; Secondary 11B37.

\noindent \emph{Keywords: } 
curling number, Gijswijt sequence, sequences, conjectures, Fine-Wilf theorem.

\bigskip
\hrule
\bigskip

\noindent (Concerned with sequences
\seqnum{A027375},
\seqnum{A090822},
\seqnum{A122536},
\seqnum{A135491},
\seqnum{A160766},
\seqnum{A216730},
\seqnum{A216813},
\seqnum{A216950},
\seqnum{A216951},
\seqnum{A216955},
\seqnum{A216955},
\seqnum{A217208},
\seqnum{A217209},
\seqnum{A217437},
\seqnum{A217943},
\seqnum{A218869},
\seqnum{A218870},
\seqnum{A218875}, and
\seqnum{A218876}.)

\bigskip
\hrule
\bigskip


\begin{thebibliography}{99}

\bibitem{AS2003}
J.-P.~Allouche and J.~Shallit,
{\em Automatic Sequences: Theory, Applications, Generalizations}, 
Cambridge Univ. Press, 2003.

\bibitem{GIJ}
F.~J.~van~de~Bult, D.~C.~Gijswijt, J.~P.~Linderman, N.~J.~A.~Sloane, and
A.~R.~Wilks,
A slow-growing sequence defined by an unusual recurrence,
{\em J. Integer Sequences},
{\bf 10} (2007), \#07.1.2.

\bibitem{BS2009}
B.~Chaffin and N.~J.~A.~Sloane,
The Curling Number Conjecture,
 \url{http://arxiv.org/abs/0912.2382}.

\bibitem{FW65}
N.~J.~Fine and H.~S.~Wilf,
Uniqueness theorems for periodic sequences,
{\em Proc. Amer. Math. Soc.},
{\bf 16} (1965), 109--114.

\bibitem{GR61}
E.~N.~Gilbert and J.~Riordan, 
Symmetry types of periodic sequences, 
{\em Illinois J. Math.}, 
{\bf 5} (1961), 657--665.

\bibitem{LS67}
A.~Lentin and M.~P.~Sch\"{u}tzenberger,
A combinatorial problem in the theory
of free monoids,
pp. 128--144 of
R.~C.~Bose and T.~A.~Dowling, eds., 
{\em Combinatorial Mathematics and its Applications},
Univ. North Carolina Press, Chapel Hill, 1969.

\bibitem{Loth}
M.~Lothaire,
{\em Combinatorics on Words},
Addison-Wesley, 1983.

\bibitem{LuVa}
A.~de~Luca and S.~Varricchio,
{\em Finiteness and Regularity in Semigroups and Formal Languages},
Springer, 1998.

\bibitem{OEIS}
The OEIS Foundation Inc.,
{\em The On-Line Encyclopedia of Integer Sequences},
\url{http://oeis.org}.

\end{thebibliography}
\end{document}